\documentclass{amsart}
\usepackage{amssymb}
\usepackage{graphics}
\usepackage{tikz}

\theoremstyle{plain}

\numberwithin{equation}{section}

\newtheorem{proposition}[equation]{Proposition}
\newtheorem{theorem}{Theorem}
\Huge
\newtheorem{lemma}[equation]{Lemma}
\newtheorem{corollary}[equation]{Corollary}
\theoremstyle{definition}
\newtheorem{remark}[equation]{Remark}
\newtheorem{example}[equation]{Example}

\newtheorem*{theorem 1}{Theorem 1}
\newtheorem*{theorem 2}{Theorem 2}
\newtheorem*{theorem 3(1)}{Theorem 3(1)}
\newtheorem*{theorem 3(2)}{Theorem 3(2)}
\newtheorem*{theorem 3(3)}{Theorem 3(3)}
\newtheorem*{theorem 4(1')}{Theorem 4(1')}
\newtheorem*{theorem 4(2)}{Theorem 4(2)}
\newtheorem*{theorem 4(3)}{Theorem 4(3)}
\newtheorem*{theorem 5}{Theorem 5}
\def    \R  {{\Bbb R}}
\def    \Z  {{\Bbb Z}}

\def    \CP {{\Bbb {CP}}}

\def    \C  {{\Bbb C}}
\def    \N  {{\Bbb N}}

\def    \index  {{\operatorname{index}}}
\def    \Tilde  {\widetilde}

\def    \ut     {\Tilde{u}}
\def    \Gt     {\Tilde{G}}

\begin{document}
\title[Hamiltonian circle action in dimension 10]{Hamiltonian $S^1$-manifolds which are like a coadjoint orbit of $G_2$}

\author{Hui Li} 
\address{School of mathematical Sciences\\
        Soochow University\\
        Suzhou, 215006, China.}
        \email{hui.li@suda.edu.cn}

\thanks{2020 classification. 53D05, 53D20, 55N25, 57R20}

\begin{abstract}
Consider a compact symplectic manifold of dimension $2n$ with a Hamiltionan circle action. Then there are at least $n+1$ fixed points.  Motivated by recent works on the case that the fixed point set consists of precisely $n+1$ isolated points, this paper studies a Hamiltonian $S^1$ action on a $10$-dimensional compact symplectic manifold with exactly 6 isolated fixed points. We study the relations of the following data: the first Chern class of the manifold, 
the largest weight of the action, all the weights of the action,
the total Chern class of the manifold, and the integral cohomology ring of the manifold. 
We show how certain data can determine the others and show the similarities  of these data with those of a coadjoint orbit of the exceptional Lie group $G_2$, equipped with a Hamiltonian action of a subcircle of the maximal torus.
\end{abstract}

 \maketitle

\section{introduction}
 
Let the circle $S^1$ act on a compact  $2n$-dimensional symplectic manifold $(M, \omega)$ with moment map. The moment map is a perfect Morse-Bott function whose critical set is precisely the fixed point set $M^{S^1}$ of the $S^1$-action.
Since the even Betti numbers of $M$, $b_{2i}(M)\geq 1$ for
all $0\leq 2i\leq 2n$, $M^{S^1}$ contains
at least $n+1$ points. When $M^{S^1}$ consists of exactly $n+1$ isolated points, we say that the action has {\bf minimal} isolated fixed points.

For a compact symplectic manifold $(M, \omega)$ of dimension $2n$ with a symplectic $S^1$-action, if the fixed points are isolated, a neighborhood of each fixed point $P$ is $S^1$-equivariantly diffeomorphic to a neighborhood of the tangent space $T_P M$ with an $S^1$-linear action. The space $T_P M$ splits into $n$ copies of $\C$, on each of which $S^1$ acts by multiplication by $\lambda^{w_i}$, where $\lambda\in S^1$ and $w_i\in\Z$ for $i=1, \cdots, n$. The integers $w_i$'s are well defined, and are called the {\it  weights} of the $S^1$-action at $P$.  In this paper,  when we say  {\bf the largest weight} or {\bf the largest weight on $M$} and there is no other specification, we mean the largest of the weights of the $S^1$-action at all the fixed points on $M$. If $W^+$ and $W^-$ are respectively the set of all the positive weights and all the negative weights at all the fixed points on $M$, then it is known that  $W^- = - W^+$  (see e.g. \cite{{H}, {L'}}).

The examples we know of compact symplectic manifolds admitting  Hamiltonian circle actions with minimal isolated fixed points are the complex projective spaces $\CP^n$ with $n\geq 0$, the Grassmannian of oriented $2$-planes in $\R^{n+2}$, denoted $\Gt_2(\R^{n+2})$ with $n\geq 3$ odd, $V_5$ and $V_{22}$ in dimension $6$, and a 10-dimensional coadjoint orbit of the rank 2 exceptional Lie group $G_2$. When the dimension of the manifolds gets higher, the weights at the fixed points are more difficult to handle, this makes the classification of the weights and the invariants of the manifolds more difficult. For example, in dimension 4, known results tell us that $\CP^2$ is the only such manifolds up to  $S^1$-equivariant symplectomorphism  (\cite{{Ka}, {J}}). In dimension $6$, in terms of integral cohomology rings and total Chern classes, there are exactly
$4$ kinds, those of $\CP^3$, $\Gt_2(\R^5)$, $V_5$ and $V_{22}$ (\cite{{T}, {Mc}}). In dimension 8, 
with the help of computer programs, the work \cite{GS} narrows down to the integral cohomology ring and total Chern class of $\CP^4$ and other possible  kinds, and the work \cite{JT} further shows that there is only the integral cohomology ring and total Chern class of $\CP^4$. For an arbitrary  $2n$-dimensional  manifold $(M, \omega)$, by our Corollary~\ref{1<n+1}, we can see that
$$c_1(M)=kx \,\,\,\mbox{with $k\in\N$ and $1\leq k\leq n+1$},$$
where $x\in H^2(M; \Z)$ is a generator. In \cite{L}, for the cases when 
$c_1(M)=(n+1)x$ and $c_1(M)=nx$, the author shows that the integral cohomology ring and total Chern class of the manifold respectively correspond uniquely to those of $\CP^n$, and of $\Gt_2(\R^{n+2})$ with $n\geq 3$ odd. In all the cases mentioned here, the weights of the $S^1$-action at the fixed points are determined in the mentioned works and are standard as in the known examples.

 In this paper, we consider 10-dimensional compact Hamiltonian $S^1$-manifolds with 6 isolated fixed points. As mentioned, $\CP^5$ and $\Gt_2(\R^7)$ are such kinds of manifolds, and they are studied in \cite{L}. Here we are not studying all the 10-dimensional such manifolds,
but consider only those such manifolds which have some common feature with a 10-dimensional coadjoint orbit of $G_2$ --- the example we know. We will use $\bf {\mathcal O}$ to denote this coadjoint orbit throughout the paper. We note that the other 10-dimensional coadjoint orbit of $G_2$ is diffeomorphic to $\Gt_2(\R^7)$ (see \cite{KT})\footnote{I thank R. Sjamaar for sending me this reference from where I know this fact.}. The maximal torus $T^2$ of $G_2$ acts on $\mathcal O$ in a Hamiltonian way 
with $6$ isolated fixed points. From this Hamiltonian $T^2$-manifold, we can create a compact Hamiltonian $S^1$-manifold of dimension 10 with 6 isolated fixed points, 
 other than $\CP^5$ and $\Gt_2(\R^7)$. In this paper, not as $\mathcal O$,  for the compact Hamiltonian $S^1$-manifolds we consider, we do not assume that the manifolds have more symmetry than an $S^1$-action. Manifolds with  $S^1$-actions appear more rigid than manifolds with higher dimensional torus actions, and they appear more difficult to study in some aspects.
In this work, we will study the interplay between the topological and geometric data --- the first Chern class, the total Chern class, and the integral cohomology ring of the manifold, and data related to the circle action --- the largest weight and all the weights of the $S^1$-action at all the fixed points. The study 
reveals close relationships between these data.

To be more concrete, let  $(M, \omega)$ be a compact Hamiltonian $S^1$-manifold of dimension 10 with moment map $\phi\colon M\to\R$, assume the fixed point set $M^{S^1}$ consists of 5 points, i.e., $M^{S^1}=\{P_0,  P_1, \cdots, P_5\}$, and $[\omega]$ is a primitive integral class. Then by Lemma~\ref{k|}, 
$$\phi(P_i)-\phi(P_j)\in\Z \,\,\,\mbox{for all $i, j$},$$ 
and by Lemma~\ref{order}, 
$$\phi(P_0) < \phi(P_1) < \cdots < \phi(P_5).$$ 
The assumption on the number of fixed points immediately implies two things,
one is that $H^2(M; \Z)=\Z$ (Lemma~\ref{order}), so we can scale the symplectic class $[\omega]$ to be primitive integral, and the other is that $M$ is connected.  We will state our main results below. 

Examples~\ref{o} and \ref{c1O} on the coadjoint orbit $\mathcal O$ has 
$c_1(\mathcal O)=3[\omega]$ and
has largest weight 5. Examples~\ref{CP5} and \ref{grass} on $\CP^5$ and $\Gt_2(\R^7)$ 
respectively have $c_1(\CP^5)=6[\omega]$ and  $c_1(\Gt_2(\R^7))=5[\omega]$, and both have largest weight at least 5, see also Examples~\ref{c1CP5} and \ref{c1grass}.
First, in Theorem~\ref{thm1}, we give the range of $c_1(M)$ and of the largest weight on $M$.

\begin{theorem}\label{thm1}
Let  $(M, \omega)$ be a compact $10$-dimensional Hamiltonian $S^1$-manifold. Assume $M^{S^1}=\{P_0,  P_1, \cdots, P_5\}$ and $[\omega]$ is a primitive integral class. Then 
$c_1(M)=k[\omega]$ with $k\in\N$ and $1\leq k\leq 6$, and the largest weight is at least 5.
\end{theorem}

Next, we consider those manifolds which have some common feature with $\mathcal O$. As mentioned, $c_1(\mathcal O)=3[\omega]$.
When the largest weight is the smallest, i.e.,  $5$, Theorem~\ref{thm2} discloses the equivalence of the topological condition $c_1(M)=3[\omega]$ to the 2 conditions on the largest weight of the circle action. When the largest weight is larger than 5, then this equivalence as in Theorem~\ref{thm2} fails, see Remark~\ref{remw} for an example. This special feature of the case when the largest weight is 5 makes us to single it out and study it independantly.

\begin{theorem}\label{thm2}
Let  $(M, \omega)$ be a compact $10$-dimensional Hamiltonian $S^1$-manifold with moment map $\phi\colon M\to\R$. Assume $M^{S^1}=\{P_0,  P_1, \cdots, P_5\}$ and $[\omega]$ is a primitive integral class. 
 Assume the largest weight is $5$. Then the following 3 conditions are equivalent.
\begin{enumerate}
\item $c_1(M)=3[\omega]$.
\item The largest weight is from $P_0$ to $P_5$\footnote{See Section 3 for definition.} and is equal to $\frac{1}{2}\big(\phi(P_5)-\phi(P_0)\big)$.
\item  The largest weight is from $P_1$ to $P_3$, is from $P_2$ to $P_4$,  and is equal to $\phi(P_3)-\phi(P_1) =\phi(P_4)-\phi(P_2)$; moreover, $\phi(P_1)-\phi(P_0)=\phi(P_5)-\phi(P_4)$.
\end{enumerate}
Under any one condition,  $(2)$ and $(3)$ are the only ways the largest weight 5 occurs and $|\pm 5|$ occurs with multiplicity one at the corresponding fixed points. 
\end{theorem}

When the largest weight is the smallest, i.e.,  is $5$, starting from Condition (2) of Theorem~\ref{thm2}, i.e., the largest weight is from $P_0$ to $P_5$ and is equal to $\frac{1}{2}\big(\phi(P_5)-\phi(P_0)\big)$, we can obtain strong conclusions as in Theorem~\ref{thm3} --- this single weight condition determines the global invariants of the manifold, namely the integral cohomology ring and the total Chern class, moreover it determines all the weights of the circle action. In Theorem~\ref{thm3}, by Theorem~\ref{thm2}, instead of using the largest weight condition, we can instead use the topological condition $c_1(M)=3[\omega]$ --- the first Chern class alone determines the above mentioned global invariants of the manifold and all the weights of the circle action.

\begin{theorem}\label{thm3}
Let  $(M, \omega)$ be a compact $10$-dimensional Hamiltonian $S^1$-manifold with moment map 
$\phi\colon M\to\R$. Assume $M^{S^1} = \{P_0,  P_1, \cdots, P_5\}$ and $[\omega]$ is a primitive integral class.   
If the largest weight is from $P_0$ to $P_5$ and is equal to $5=\frac{1}{2}\big(\phi(P_5)-\phi(P_0)\big)$, then the following 3 things hold.
\begin{enumerate}
\item The integral cohomology ring $H^*(M; \Z)$ is generated by the elements
$$1,\,\,\, [\omega], \,\,\,\frac{1}{3}[\omega]^2,\,\,\, \frac{1}{6}[\omega]^3,\,\,\, \frac{1}{18}[\omega]^4, \,\,\,\frac{1}{18}[\omega]^5,$$
and is isomorphic to the integral cohomology ring $H^*(\mathcal O; \Z)$.
\item The total Chern class
$$c(M) = 1+3[\omega]+\frac{1}{3}\big(13[\omega]^2 + 11[\omega]^3 + 5[\omega]^4 + [\omega]^5\big)$$
$$=1 + 3\alpha_1 +13 \alpha_2 +22\alpha_3 +30 \alpha_4 +6 \alpha_5,$$
where $\alpha_i\in H^{2i}(M; \Z)$ is a generator for $1\leq i\leq 5$.
The total Chern classes $c(M)$ and $c(\mathcal O)$ are isomorphic.
\item The sets of weights at the fixed points are the same as those of a subcircle action of a maximal torus of $G_2$ on $\mathcal O$ (as in Example~\ref{o}); in particular, there is exactly one weight between any pair of fixed points. The precise sets of weights at all the fixed points on $M$ are as follows:
$$P_0\colon \{1, 4, 2, 3, 5\}, \,\,\, P_5\colon -\{1, 4, 2, 3, 5\}, $$
$$P_1\colon \{-1, 1, 5, 4, 3\},\,\,\, P_4 \colon -\{-1, 1, 5, 4, 3\},$$
$$P_2\colon \{-4, -1, 1, 5, 2\},\,\,\, P_3 \colon -\{-4, -1, 1, 5, 2\}.$$ 
\end{enumerate}
\end{theorem}

Next, we consider the more general case on the largest weight. In this case, we need to impose more complex condition on the largest weight to determine similar data as in Theorem~\ref{thm3}. 
Essentially we put similar conditions as in (2) and (3) in Theorem~\ref{thm2} together to obtain 
the results in Theorem~\ref{thm4}.
The condition we impose is still similar to that in Example~\ref{o} for the coadjoint orbit 
$\mathcal O$. We remark that without such a condition, we do not have the conclusion as in Theorem~\ref{thm4}, see Remark~\ref{remw}.

\begin{theorem}\label{thm4}
Let  $(M, \omega)$ be a compact $10$-dimensional effective Hamiltonian $S^1$-manifold with moment map 
$\phi\colon M\to\R$. Assume $M^{S^1} = \{P_0,  P_1, \cdots, P_5\}$ and $[\omega]$ is a primitive integral class.  Assume  the largest weight  is  from $P_0$ to $P_5$, from $P_1$ to $P_3$, and from $P_2$ to $P_4$ and is equal to $\frac{1}{2}\big(\phi(P_5)-\phi(P_0)\big) = \phi(P_3)-\phi(P_1)=\phi(P_4)-\phi(P_2)$.  Assume $\phi(P_1)-\phi(P_0)=\phi(P_5)-\phi(P_4)$. 
 Then the following 3 things hold.
\begin{enumerate}
\item The integral cohomology ring $H^*(M; \Z)$ is generated by the elements
$$1,\,\,\, [\omega], \,\,\,\frac{1}{3}[\omega]^2,\,\,\, \frac{1}{6}[\omega]^3,\,\,\, \frac{1}{18}[\omega]^4, \,\,\,\frac{1}{18}[\omega]^5,$$
and is isomorphic to the integral cohomology ring $H^*(\mathcal O; \Z)$.
\item The total Chern class
$$c(M) = 1+3[\omega]+\frac{1}{3}\big(13[\omega]^2 + 11[\omega]^3 + 5[\omega]^4 + [\omega]^5\big)$$
$$=1 + 3\alpha_1 +13 \alpha_2 +22\alpha_3 +30 \alpha_4 +6 \alpha_5,$$
where $\alpha_i\in H^{2i}(M; \Z)$ is a generator for $1\leq i\leq 5$.
The total Chern classes $c(M)$ and $c(\mathcal O)$ are isomorphic.
\item There is exactly one weight between any pair of fixed points; the precise sets of weights at all the fixed points are as follows:
$$P_0\colon \Big\{a,\, a+c,\, a+\frac{c}{3},\, a+\frac{2}{3}c,\, 2a+ c\Big\},\,\,\,  P_5\colon\Big\{-2a-c,\, -a-\frac{2}{3}c,  \, -a-\frac{c}{3},\, -a-c,\, -a\Big\},$$
$$P_1\colon \Big\{-a,\, \frac{c}{3}, \, 2a+ c,\, a+c,\, a+\frac{2}{3}c\Big\}, \,\,\,  P_4\colon \Big\{-a-\frac{2}{3}c,\, -a-c,\, -2a-c,\, - \frac{c}{3},\, a\Big\},$$
$$P_2\colon \Big\{-a-c,\, - \frac{c}{3},\, a, \, 2a+c,\, a+\frac{c}{3}\Big\}, \,\,\,  P_3 \colon \Big\{-a-\frac{c}{3},\, -2a-c,\, -a,\, \frac{c}{3}, \, a+c\Big\},$$
where $a = \phi(P_1)-\phi(P_0)$ and $c = \phi(P_2)-\phi(P_1)$.
For $a=1$ and $c=3$, we get exactly the data as those in Theorem~\ref{thm3} (3). 
\end{enumerate}
\end{theorem}

%As we mentioned, $\mathcal O$ admits a Hamiltonian $T^2$-action. For our Hamiltonian $S^1$-%manifold  $(M, \omega)$, if it admits a bigger dimensional Hamiltonian torus action, and if we are %interested in knowing $H^*(M; \Z)$ and $c(M)$, we can consider a suitable subcircle action, the %current work provides an easy way of getting $H^*(M; \Z)$ and $c(M)$ using the subcircle action %instead of the torus action. 

The compact Hamiltonian $S^1$-manifolds of dimension 10 with 6 fixed points we study here have the same even Betti numbers as 
$\CP^5$ but have integral cohomology rings different from that of $\CP^5$ (Theorems~\ref{thm3} and \ref{thm4}).
For the study of compact Hamiltonian $S^1$-manifolds 
whose even Betti numbers are the same as those of $\CP^n$, %or close to those of $\CP^n$. 
but may not restrict to isolated fixed points, see \cite{{T}, {LT}, {LOS}, {L}, {L2}}. 
The current work discloses close relationships between the global geometric and topological data, the first Chern class, the total Chern class and the integral cohomology ring of the manifolds and the local data related to the circle action under the context of Hamiltonian actions. 
The results obtained in this paper, besides standing on their own right,  build a foundation for further classification of the manifolds in various categories, for instance, in the symplectic, complex, or K\"ahler categories, noticing that the example $\mathcal O$ is a K\"ahler manifold; for this line of works, see for instance \cite{Mc, LOS, L2}. In the K\"ahler category, the problem is also related to the classification of {\it fake complex projective spaces}, namely compact K\"ahler manifolds which have the same Betti numbers as $\CP^n$ but are different from $\CP^n$. 

\medskip

Now we summarize the structure of the paper. In Section 2, we give the 10-dimensional example $\mathcal O$, we also give the more known examples $\CP^5$ and $\Gt_2(\R^7)$. In Section 3, we state with minor proof some basic results for Hamiltonian $S^1$-actions, building a foundation for the next sections. In Section 4, we give two ways of expressing the first Chern class of the manifold. In Section 5, we study the relation between the first Chern class of the manifold and the largest weight of the $S^1$-action, and prove Theorems~\ref{thm1} and \ref{thm2}.  Section 6 has 4 parts. In part 6.1, when the largest weight is 5, we use the weight condition to determine all the weights of the action, proving Theorem~\ref{thm3} (3). In part 6.2, for the general case, we use the weight condition to determine all the weights of the action, proving Theorem~\ref{thm4} (3). In part 6.3, we use the weight condition to determine the integral cohomology ring of the manifold, proving Theorem~\ref{thm3} (1) and Theorem~\ref{thm4} (1). In part 6.4, we use the weight condition to determine the total Chern class of the manifold, proving Theorem~\ref{thm3} (2) and Theorem~\ref{thm4} (2).

%\subsubsection*{Acknowledgement}  

\section{the example --- a coadjoint orbit of $G_2$}

In this section, we give an explicit example, a coadjoint orbit of $G_2$ --- a compact 10-dimensional Hamiltonian $S^1$-manifold  with exactly 6 isolated fixed points.
We also include the two more known examples $\CP^5$ and $\Gt_2(\R^7)$.

\begin{example}\label{o}
In \cite{Mo}(p42-44), Morton gives a GKM-manifold, it is a $10$-dimensional coadjoint orbit 
$\mathcal O$ of the 14-dimensional rank 2 Lie group $G_2$, passing through a point on a wall of the closed Weyl chamber. The maximal torus $T^2$ of $G_2$ acts on $\mathcal O$. The $T^2$-action is Hamiltonian, and up to scale and $SL(2; \Z)$ transformation, the moment polytope UVWXYZ has the following vertices: 
$$U=(-1, -2),\,\,\, V=(-2, -1),\,\,\, W=(-1, 1),\,\,\, X=(1, 2),\,\,\, Y=(2,1),\,\,\, Z=(1, -1).$$
They are exactly the $T^2$ moment map images of the 6 isolated fixed points of the $T^2$-action.
 The weights of the $T^2$-action of the edges are:
$$w(UV)=(-1, 1), w(UZ)=(2,1), w(UW)=(0, 1), w(UY)=(1, 1), w(UX)=(1, 2),$$
$$w(VZ)=(1, 0), w(VW)=(1, 2), w(VY)=(2, 1), w(VX)=(1, 1),$$
$$w(ZW)=(-1, 1), w(ZY)=(1, 2), w(ZX)=(0, 1),$$
$$w(WY)=(1, 0), w(WX)=(2, 1),$$
$$w(YX)=(-1, 1).$$
(The picture is a GKM graph in \cite{Mo}. The graph has the following relation with the vertices of the moment polytope. Let $\psi$ be the $T^2$ moment map, and $\ut$ be the equivariant extension of 
the symplectic class $[\omega]$ similarly as in Lemma~\ref{ut} for $\psi$. Since $H^2(\mathcal O; \R)=\R$ (generated by $[\omega]$), we have $c_1^{S^1}(M)=k\ut$ with $k\in\R$, then $c_1^{S^1}(M)|_{P} =\Gamma_P = - k\psi(P)$ for any fixed point $P$ of $T^2$, where $\Gamma_P$ is the sum of weights at $P$.)

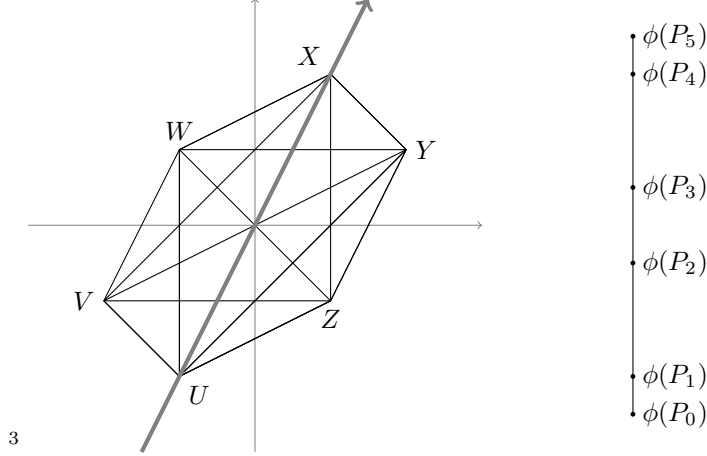
\footnote{I thank Chen He for telling me how to draw these pictures.}
\begin{tikzpicture}
\draw [help lines, ->] (-3,0)--(3,0);
\draw [help lines, ->] (0,-3)--(0,3);
\draw (-1,-2)--(-2,-1)--(-1,1)--(1,2)--(2,1)--(1,-1)--(-1,-2);
\draw (-1, -2)--(-1, 1);
\draw (-1, -2)--(1, 2);
\draw (-1, -2)--(2, 1);
\draw (-1, -2)--(1, -1);
\draw (-2, -1)--(1, 2);
\draw (-2, -1)--(2, 1);
\draw (-2, -1)--(1, -1);
\draw (-2, -1)--(-1, -2);
\draw (-1, 1)--(1, 2);
\draw (-1, 1)--(2, 1);
\draw (-1, 1)--(1, -1);
\draw (-1, 1)--(-1, -2);
\draw (1, 2)--(2, 1);
\draw (1, 2)--(1, -1);
\draw (1, 2)--(-1, -2);
\draw (2, 1)--(1, -1);
\draw (2, 1)--(-1, -2);
\draw (1, -1)--(-1, -2);
\node[below right] at(-1, -2){$U$};
\node[left] at(-2, -1){$V$};
\node[above] at(-1, 1){$W$};
\node[above left] at (1, 2){$X$};
\node[right] at (2, 1){$Y$};
\node[below] at (1, -1){$Z$};
\draw [help lines, ->] [ultra thick] (-1.5,-3)--(1.5, 3);
\draw (5,-2.5)--(5, -2)--(5,-0.5)--(5,0.5)--(5, 2)--(5,2.5);
\draw[fill] (5, -2.5) circle [radius=0.025];
\draw[fill] (5, -2) circle [radius=0.025];
\draw[fill] (5, -0.5) circle [radius=0.025];
\draw[fill] (5, 0.5) circle [radius=0.025];
\draw[fill] (5, 2) circle [radius=0.025];
\draw[fill] (5, 2.5) circle [radius=0.025];
\node[right] at (5, -2.5){$\phi(P_0)$};
\node[right] at (5, -2){$\phi(P_1)$};
\node[right] at (5, -0.5){$\phi(P_2)$};
\node[right] at (5, 0.5){$\phi(P_3)$};
\node[right] at (5, 2){$\phi(P_4)$};
\node[right] at (5, 2.5){$\phi(P_5)$};
\end{tikzpicture}

We take the subcircle $S^1\subset T^2$ generated by the direction $UX = (1, 2)$. Let $\phi$ be the moment map of the circle action on $\mathcal O$.
The image of $\phi$ is the projection of the polytope to the line segment $UX$. 
The 6 vertices of the polytope respectively project to 6 points on
the line, which correspond respectively to the $\phi$ images of the 6 fixed points of the circle action on $\mathcal O$. Let $P_0$, $P_1$, $\cdots$, $P_5$ be the $6$ fixed points on $\mathcal O$ of the circle action, in the order that 
$$\phi(P_0) < \phi(P_1) <\cdots < \phi(P_5).$$
We project the weights of the $T^2$-action to the direction $(1, 2)$, we get the weights of the 
$S^1$-action at the 6 fixed points of the $S^1$-action and their relations with the moment map values as follows.
$$P_0\colon  \{1, 4, 2, 3, 5\}$$
$$=\Big\{\phi(P_1)-\phi(P_0), \phi(P_2)-\phi(P_0), \frac{1}{3}\big(\phi(P_3)-\phi(P_0)\big), \frac{1}{3}\big(\phi(P_4)-\phi(P_0)\big),
\frac{1}{2}\big(\phi(P_5)-\phi(P_0)\big)\Big\}.$$
$$P_1\colon \{-1, 1, 5, 4, 3\}$$
$$=\Big\{\phi(P_0)-\phi(P_1), \frac{1}{3}\big(\phi(P_2)-\phi(P_1)\big), \phi(P_3)-\phi(P_1), \frac{1}{2}\big(\phi(P_4)-\phi(P_1)\big), 
\frac{1}{3}\big(\phi(P_5)-\phi(P_1)\big)\Big\}.$$
$$P_2 \colon \{-4, -1, 1, 5, 2\}$$
$$=\Big\{\phi(P_0)-\phi(P_2), \frac{1}{3}\big(\phi(P_1)-\phi(P_2)\big), \frac{1}{2}\big(\phi(P_3)-\phi(P_2)\big), \phi(P_4)-\phi(P_2), 
\frac{1}{3}\big(\phi(P_5)-\phi(P_2)\big)\Big\}.$$
$$P_3\colon \{-2, -5, -1, 1, 4\}$$
$$=\Big\{\frac{1}{3}\big(\phi(P_0)-\phi(P_3)\big), \phi(P_1)-\phi(P_3), \frac{1}{2}\big(\phi(P_2)-\phi(P_3)\big), \frac{1}{3}\big(\phi(P_4)-\phi(P_3)\big), \phi(P_5)-\phi(P_3)\Big\}.$$
$$P_4\colon \{-3, -4, -5, -1, 1\}$$
$$=\Big\{\frac{1}{3}\big(\phi(P_0)-\phi(P_4)\big), \frac{1}{2}\big(\phi(P_1)-\phi(P_4)\big), \phi(P_2)-\phi(P_4), \frac{1}{3}\big(\phi(P_3)-\phi(P_4)\big), \phi(P_5)-\phi(P_4)\Big\}.$$
$$P_5\colon \{-5, -3, -2, -4, -1\}$$
$$=\Big\{\frac{1}{2}\big(\phi(P_0)-\phi(P_5)\big), \frac{1}{3}\big(\phi(P_1)-\phi(P_5)\big), \frac{1}{3}\big(\phi(P_2)-\phi(P_5)\big), \phi(P_3)-\phi(P_5), \phi(P_4)-\phi(P_5)\Big\}.$$
\end{example}

We now give the other 2 known examples of Hamiltonian $S^1$-manifolds of dimension 10 with 6 fixed points, $\CP^5$ and $\Gt_2(\R^7)$.
\begin{example}\label{CP5}
Consider $\CP^5$ with the $S^1$-action:
$$\lambda\cdot [z_0, z_1, \cdots, z_5] = [z_0, \lambda^a z_1,\lambda^{a+b}z_2,\lambda^{a+b+c}z_3,   \lambda^{a+b+c+d}z_4, \lambda^{a+b+c+d+e}z_5],$$
where $a, b, c, d, e\in\N$.
This action has $6$ fixed points, $P_i = [0, \cdots, z_i, \cdots, 0]$, where $i=0, \cdots, 5$, respectively with  moment map $\phi$ values $\phi(P_i)$'s:
$$0, \,a,\, a+b,\, a+b+c,\, a+b+c+d,\, a+b+c+d+e.$$ 
The set of weights at $P_i$ is $\{\phi(P_j)-\phi(P_i)\}_{j\neq i}$, with $i=0, \cdots, 5$. 
\end{example}

\begin{example}\label{grass}
Let $\Gt_2(\R^7)$ be the  Grassmanian of oriented $2$-planes in $\R^7$, and let the  $S^1$-action on $\Gt_2(\R^7)$ be induced from the $S^1$-action on $\R^7$  
given by
$$\lambda\cdot (t, z_0, z_1, z_2) = (t, \lambda^{a+b+\frac{c}{2}}z_0, \lambda^{b+\frac{c}{2}}z_1, \lambda^{\frac{c}{2}} z_2),$$
where $a, b\in\N$ and $c\in 2\N$. This action has $6$ fixed points, $P_i$ with $i=0, \cdots, 5$,  where  $P_i$ and $P_{5-i}$ are the plane $(0, \cdots, z_i, \cdots, 0)$  with two different orientations. 
The fixed points have moment map $\phi$ values $\phi(P_i)$'s:
 $$-a-b -\frac{c}{2},\, -b-\frac{c}{2},\, -\frac{c}{2},\, \frac{c}{2},\, b+\frac{c}{2},\,  a +b +\frac{c}{2}.$$
The set of weights at $P_i$ is $\{\phi(P_j)-\phi(P_i)\}_{j\neq i, 5-i}\cup \frac{1}{2}\big(\phi(P_{5-i})-\phi(P_i)\big)$, with $i=0, \cdots, 5$.
\end{example}

\section{some basic results} 

In this section, we state with minor proof some basic results which are important for the study in the next sections. 

\smallskip

Let $M$ be a smooth
 $S^1$-manifold. The {\bf equivariant cohomology} of $M$ in a coefficient ring $R$ is
 $H^*_{S^1}(M; R) = H^*(S^{\infty}\times_{S^1} M; R)$, where
 $S^1$ acts on $S^{\infty}$ freely. If $p$ is a point, then $H^*_{S^1}(p; R)= H^*(\CP^{\infty}; R)=R[t]$, where $t\in H^2(\CP^{\infty}; R)$ is a generator.
 If $S^1$ acts on $M$ trivially, then $H^*_{S^1}(M; R)= H^*(M; R)\otimes R[t] =  H^*(M; R)[t]$. The projection map $\pi\colon S^{\infty}\times_{S^1} M\to \CP^{\infty}$ induces a pull back map
$\pi^*\colon H^*(\CP^{\infty}) \to H^*_{S^1}(M)$,
so that $H^*_{S^1}(M)$ is an $H^*(\CP^{\infty})$-module.

Let $(M, \omega)$ be a compact $2n$-dimensional symplectic $S^1$-manifold with isolated fixed points.  The equivariant total Chern class of $M$ can be expressed as
 $$c^{S^1}(M) = 1 + c_1^{S^1}(M) + \cdots +  c_n^{S^1}(M)\in H^*_{S^1}(M; \Z),$$
where $c_i^{S^1}(M)\in  H^{2i}_{S^1}(M; \Z)$ is the $i$-th equivariant Chern class of $M$.
Let $P$ be a fixed point, and let $\{w_1,  \cdots, w_n\}$ be the set of weights at $P$.
Then the restriction of $c^{S^1}(M)$ to $P$ is
$$c^{S^1}(M)|_P = 1 +  \sum_{i=1}^n c_i^{S^1}(M)|_P = 1  + \sum_{i=1}^n\sigma_i(w_1, \cdots,  w_n) t^i,$$
 where $\sigma_i(w_1, \cdots, w_n)$ is the $i$-th symmetric polynomial in the weights at $P$. The restriction of $c^{S^1}(M)$ to ordinary cohomology is the  total Chern class of $M$:
$$c(M) = 1 + c_1(M) + \cdots + c_n(M)\in H^*(M; \Z),$$
where $c_i(M) \in H^{2i}(M; \Z)$ is the $i$-th Chern class of $M$.

 Let $(M, \omega)$ be a compact Hamiltonian $S^1$-manifold with isolated fixed points. 
The moment map is a Morse function, its critical set is the set of fixed points.  At a fixed point $P$, if $\lambda_P$ is the number of negative weights and $\lambda_P^+$ is the number of positive weights, then the {\bf Morse index of $P$}  is $\bf 2\lambda_P$, and the {\bf Morse coindex of  $P$} is $2\lambda_P^+$.

Now we start to state the results. First, the symplectic class $[\omega]$ of a Hamiltonian $S^1$-manifold can be extended to an equivariant cohomology class $\ut$ as in Lemma~\ref{ut}, and based on this, we can obtain Lemma~\ref{k|}.

\begin{lemma}\cite{LT}\label{ut}
Let the circle act  on a compact symplectic manifold $(M,\omega)$
with moment map $\phi \colon M \to \R$.  Then there exists $\ut =[\omega -\phi t]\in H_{S^1}^2(M;\R)$ such that for any fixed point set component $F$,
$$\ut|_F =[\omega|_F]  - \phi(F)t.$$
If $[\omega]$ is an integral class, then $\ut$ is an integral class.  
\end{lemma}

For an $S^1$-manifold $M$,  when there exists a finite stabilizer group $\Z_k\subset S^1$ with $k > 1$, the set of points, $M^{\Z_k}$, which is pointwise fixed by $\Z_k$ but not pointwise fixed by $S^1$, is called a  {\bf $\Z_k$-isotropy submanifold}. 

\begin{lemma}\label{k|}\cite{L}
Let the circle act on a compact symplectic manifold
$(M, \omega)$ with moment map $\phi\colon M\to\R$. Assume  $[\omega]$ is an integral class. Then for any two fixed point set components $F$ and $F'$, $\phi(F) - \phi(F') \in \Z$.  If $\Z_k$ is the stabilizer group of some point on $M$,  then for any two fixed point set components $F$ and
$F'$ on the same connected component of $M^{\Z_k}$, we have $k\,|\left(\phi(F') - \phi(F)\right)$.
\end{lemma}
When $[\omega]$ is an integral class, we will often use the fact that $\phi(F) - \phi(F') \in \Z$ in Lemma~\ref{k|} without always referring to it.

As an application of the above lemmas, we have the following lemma.
\begin{lemma}\label{<}
Let the circle act on a connected compact symplectic manifold $(M, \omega)$ with moment map $\phi\colon M\to\R$. Assume $M^{S^1}$ consists of isolated points and $[\omega]$ is an integral class.
Let $P$ be a fixed point. If $-a$ is a negative weight at $P$, then $a|\big(\phi(Q)-\phi(P)\big)$ for some $\phi(Q)<\phi(P)$. Similarly, if $b$ is a positive weight at $P$, then $b|\big(\phi(Q)-\phi(P)\big)$ for some $\phi(Q) > \phi(P)$.
\end{lemma}

\begin{proof}
Since $a$ is a negative weight at $P$, $P$ is not the minimal fixed point in $M^{\Z_a}$ containing $P$, so there is a $Q$ in $M^{\Z_a}$ containing $P$ with $\phi(Q)<\phi(P)$. Then $a|\big(\phi(Q)-\phi(P)\big)$ by Lemma~\ref{k|}. The other claim follows similarly.
\end{proof}

Based on the symplectic nature of the manifold, Tolman gives a claim as follows on the index of a fixed point for a Hamiltonian $S^1$-manifold. 

\begin{lemma}\label{ind}\cite{T}
Let the circle act on a connected compact symplectic manifold
$(M, \omega)$ with isolated fixed points and with moment map $\phi\colon M\to\R$. Then for any fixed point $P$, $2\lambda_P\leq 2l$, where $l$ is the number of fixed points $Q$'s such that $\phi(Q) < \phi(P)$.
\end{lemma}

Next, for any two fixed  points on a connected component of an isotropy submanifold $M^{\Z_k}$, the sets of $S^1$-weights at them are related as follows.

\begin{lemma}\cite{T}\label{mod}
Let the circle act on a compact symplectic manifold $(M, \omega)$. Let $P$ and $Q$ be fixed points which lie on the same connected component of $M^{\Z_k}$ for some $k > 1$. Then the weights of the $S^1$-action at $P$ and $Q$ are equal modulo $k$.
\end{lemma}

A variant of Lemma~\ref{mod} is: the weights of the $S^1$-action at the normal spaces 
of the component of $M^{\Z_k}$ at $P$ and at $Q$ are equal modulo $k$. We will refer to Lemma~\ref{mod} for this fact. 

\smallskip

The next result by Jang and Tolman is a claim on the multiplicity of the smallest positive weight on the manifold. It is given for the more general category of almost complex manifolds than symplectic manifolds.
\begin{lemma}\cite{JT}\label{alcx}
Let the circle act on a closed $2n$-dimensional almost
complex manifold  with isolated fixed points. Let $w$ be the smallest
positive weight that occurs at the fixed points on $M$. Then given any 
$k \in\{0, 1, . . . , n-1\}$, the number of times the weight $-w$ occurs at
fixed points of index $2k+2$ is equal to the number of times the 
weight $+w$ occurs at fixed points of index $2k$.
\end{lemma}

Now, we start to consider compact Hamiltonian $S^1$-manifolds with minimal isolated fixed points.  We will use the following facts on the Morse indices of the fixed points and (\ref{eqorder}) so often that we do not refer to the lemma. 

\begin{lemma}\label{order}\cite{L}
Let the circle act on a compact $2n$-dimensional symplectic manifold $(M, \omega)$ with moment map
 $\phi\colon M\to\R$. Assume $M^{S^1}$ consists of $n+1$ isolated points, denoted $M^{S^1}=\{P_0,  P_1, \cdots, P_n\}$. Then the points in $M^{S^1}$ can be labelled so that $P_i$  has Morse index $2i$, and we have $H^{2i}(M; \Z) = \Z$ and $H^{2i-1}(M; \Z)=0$ for all $0\leq 2i\leq 2n$. 
Moreover, 
\begin{equation}\label{eqorder}
\phi(P_0) < \phi(P_1) < \cdots < \phi(P_n). 
\end{equation}
\end{lemma}

\medskip

In a symplectic $S^1$-manifold $(M, \omega)$ with isolated fixed points,
if $w>0$ is a weight at a fixed point $P$, $-w$ is a weight  at a fixed point $Q$, and $P$ and $Q$ are on the same connected component of
$M^{\Z_w}$, we say that {\bf $w$ is a weight from $P$ to $Q$} or {\bf there is a weight $w$ from $P$ to $Q$}. 
When the signs of $w$ at $P$ and at $Q$ are clear, we will also say that {\bf there is a weight $w$ or $-w$ between $P$ and $Q$}, or {\bf $w$ or $-w$ is a weight between $P$ and $Q$}.

\begin{lemma}\label{10}\cite{L}
Let the circle act on a compact $2n$-dimensional symplectic manifold
$(M, \omega)$ with moment map $\phi\colon M\to\R$. Assume $M^{S^1}=\{P_0, P_1, \cdots, P_n\}$ and $[\omega]$ is a {\it primitive} integral class. Then there is a weight $w = \phi(P_1)-\phi(P_0)$ from $P_0$ to $P_1$, and there is a weight $w'=\phi(P_n)-\phi(P_{n-1})$ from $P_{n-1}$ to $P_n$.
\end{lemma}

\begin{lemma}\label{other}\cite{L}
Let the circle act on a compact $2n$-dimensional symplectic manifold
$(M, \omega)$ with moment map $\phi\colon M\to\R$.  Assume   
$M^{S^1} = \{P_0, P_1, \cdots, P_n\}$ and $[\omega]$ is an integral class. Assume there is a weight $a=\phi(P_i)-\phi(P_0)$ from $P_0$ to $P_i$ for some $i\neq 0$. Then for any other weight $a'$ at $P_0$, which  could be equal to $a$, there exists $j\neq 0, i$, such that $a'$ is a weight from $P_0$ to $P_j$. A similar claim holds if we replace $P_0$ by $P_n$.
\end{lemma}

For a compact Hamiltonian $S^1$-manifold $M$ with minimal isolated fixed points, by standard method in equivariant cohomology by Kirwan (\cite{K}), 
we know that the extended classes of the equivariant Euler classes of the negative normal bundles of the fixed points form a basis of $H^*_{S^1}(M; \R)$  as an $H^*(\CP^{\infty}; \R)$-module. 
Since $H^*(M^{S^1}; \Z)$ has no torsion, the above fact is also true for equivariant cohomology in $\Z$-coefficients, and $H^*(M; \Z) = H^*_{S^1}(M; \Z)/ (t)$ (see Sec.2 of \cite{LT}), hence the restriction of a basis of $H^*_{S^1}(M; \Z)$ to ordinary cohomology is a basis of $H^*(M; \Z)$. We summarize these as follows.

\begin{proposition}\label{equibase}
Let the circle act on a compact $2n$-dimensional symplectic manifold $(M, \omega)$ with moment map $\phi\colon M\to\R$. Assume $M^{S^1}=\{P_0,  P_1, \cdots, P_n\}$. Then
as an $H^*(\CP^{\infty}; \Z)$-module,  $H^*_{S^1}(M; \Z)$ has a basis 
$\big\{\Tilde\alpha_i\in H^{2i}_{S^1}(M; \Z)\,|\, 0\leq i\leq n\big\}$ such that
$$\Tilde\alpha_i|_{P_i} = \Lambda_i^- t^i, \,\,\mbox{and}\,\,\,\, \Tilde\alpha_i|_{P_j} =0,\,\,\forall\,\, j < i,$$
where $\Lambda_i^-$ is the product of the negative weights at $P_i$.
Moreover, $\big\{\alpha_i = r(\Tilde\alpha_i) \in H^{2i} (M; \Z)\,|\, 0\leq i\leq n\big\}$  is a basis for $H^*(M; \Z)$, where
$r\colon H^*_{S^1}(M; \Z)\to H^*(M; \Z)$ is the natural restriction map.
\end{proposition}

A direct corollary of Proposition~\ref{equibase} is:

\begin{corollary}\label{cor}
Let the circle act on a compact $2n$-dimensional symplectic manifold $(M, \omega)$ with moment map $\phi\colon M\to\R$. Assume $M^{S^1}=\{P_0,  P_1, \cdots, P_n\}$. If 
$\Tilde\alpha\in H^{2i}_{S^1}(M; \Z)$ is a class such that $\Tilde\alpha|_{P_j}=0$ for all $j<i$, then
$$\Tilde\alpha =  a_i \Tilde\alpha_i \,\,\,\mbox{for some $a_i\in\Z$}.$$
\end{corollary}

\section{on the first Chern class $c_1(M)$}

In this section, we give two ways of expressing the first Chern class of the manifold. The first way is using the sums of the weights of two fixed points as in Lemma~\ref{sub}, the second way is using the largest weight between two fixed points as in Lemma~\ref{gammaij}. 

\begin{lemma}\label{sub}\cite{L}
Let the circle act on a connected compact symplectic manifold $(M, \omega)$ with moment map 
$\phi\colon M\to\R$. If $c_1(M)=k[\omega]$, then for any two fixed point set components
$F$ and $F'$,  we have 
$$\Gamma_F - \Gamma_{F'} = k\big(\phi(F')-\phi(F) \big),$$
where $\Gamma_F$ and $\Gamma_{F'}$ are respectively the sums of the weights at $F$ and $F'$.
\end{lemma}

\begin{lemma}\label{gammaij}
Let the circle act on a connected compact symplectic manifold $(M, \omega)$ with moment map $\phi\colon M\to\R$. Assume $M^{S^1}$ consists of isolated points. Let $P, Q\in M^{S^1}$, where $\index (P)=2i$ and $\index (Q) = 2j$ with $i \leq j$ and $\phi(P)<\phi(Q)$. Assume there is a weight $w >0$ from $P$ to $Q$,  $-w$ is not a weight at $P$, $-w$  has multiplicity $s$ at $Q$, and $w$ is the largest among the absolute values of all the weights at $P$ and $Q$. Then $w$ has multiplicity at least $s$ at $P$. If $c_1(M) = k [\omega]$,  then 
$$j-i +s = k\frac{\phi(Q)-\phi(P)}{w}.$$
If $[\omega]$ is an integral class, then $\frac{\phi(Q)-\phi(P)}{w}\in\Z$.
\end{lemma}

\begin{proof}
Let 
$$W_P^- = \big\{a_1, \cdots, a_i\big\}\,\,\,\mbox{and}\,\,\,  W_P^+ = \big\{a_{i+1}, \cdots, a_n\big\}$$
 be respectively the set of negative weights and positive weights at $P$, and  
$$W_Q^- = \big\{b_1, \cdots, b_j\big\}\,\,\,\mbox{and}\,\,\, W_Q^+ = \big\{b_{j+1}, \cdots, b_n\big\}$$
 be respectively the set of negative weights and positive weights at $Q$.
By Lemma~\ref{mod}, 
\begin{equation}\label{modw} 
W_P^-\cup W_P^+ = W_Q^- \cup W_Q^+  \mod w.
\end{equation}
Since there is no weight $-w$ at $P$, for (\ref{modw}) to hold, for each $-w$ at $Q$, there should be a $w$ at $P$ such that $w=-w \mod w$. So $w$ has at least multiplicity $s$ at $P$.
Assume
$$a_{n-s+1} = w, \cdots, a_n = w \in W_P^+, \,\,\,\mbox{and}\,\,\, b_{j-s+1}= -w, \cdots,  b_j = -w\in W_Q^-.$$
The rest of the proof is similar to that of \cite[Lemma 4.2]{L}.
Since $w$ is the largest among the absolute values of all the weights at $P$ and $Q$,  
up to a reordering of indices,  (\ref{modw}) can only yield the following relations:
$$a_1 = b_1, \,\,\,\cdots,\,\,\, a_m = b_m \,\,\,\mbox{for some $m$ with $0\leq m \leq i$},$$
$$a_{i+1} = b_{j+1},\,\,\, \cdots,\,\,\, a_{i+l} = b_{j+l} \,\,\,\mbox{for some $l$
with $0\leq l\leq n-j$},$$
$$a_{m+1}  = b_{j +l+1} - w, \,\,\,\cdots, \,\,\,  a_i = b_n-w,$$
$$a_{i+l+1}  = b_{m+1} + w,\,\,\, \cdots, \,\,\, a_{n-s}  = b_{j-s} + w, \,\,\,\mbox{and}$$
$$ a_{n-s+1} = b_{j-s+1} + 2 w, \,\,\, \cdots, \,\,\, a_n = b_j + 2w.$$
 Let $\Gamma_P$ and $\Gamma_Q$ be respectively the sums of weights at $P$ and at $Q$. Then
 $$\Gamma_P - \Gamma_Q = \sum_{r=1}^n a_r - \sum_{r=1}^n b_r = (j-i +s)w.$$
Combining Lemma~\ref{sub}, we obtain the first claim of the lemma.

The last claim is by  Lemma~\ref{k|}.
\end{proof}

The next lemma gives a criterion of using Lemma~\ref{gammaij} for the largest weight on $M$. The lemma is stated here according to Lemma 4.4 and its proof in \cite{L}.

\begin{lemma}\label{large}\cite{L}
Let the circle act on a compact symplectic manifold $(M, \omega)$ with moment map $\phi\colon M\to\R$. Assume $M^{S^1}$ consists of isolated points.  Let $w>0$ be the largest weight on $M$. Take a lower fixed point $P$ (relative to the value of $\phi$) so that $-w$ is not a weight at $P$, and take a closest fixed point $Q$ to $P$ with $\phi(P) < \phi(Q)$ such that there is a weight $w$ from $P$ to $Q$, then $-w$ has multiplicity $1$ at $Q$.
\end{lemma}

\begin{remark}
For the largest weight $w$ on $M$, and the choice of $P$ and $Q$ in Lemma~\ref{large},
the $s$ in Lemma~\ref{gammaij} is $1$.
In Lemma~\ref{gammaij}, if $s > 1$, then between 
$P$ and $Q$, there must be other fixed points; otherwise, the connected component of 
$M^{\Z_w}$ containing $P$ has $P$ as the minimum, and $Q$ as the next fixed point
with index bigger than 2, contradicting to Lemma~\ref{ind} applied on this component.
\end{remark}

\begin{corollary}\label{1<n+1}
Let $(M, \omega)$ be a compact Hamiltonian $S^1$-manifold $(M, \omega)$. Assume
$M^{S^1}$ consists of isolated points, $H^2(M; \Z)=\Z$ and $[\omega]$ is a primitive integral class. Then
$c_1(M) = k[\omega]$ with $k\in\N$ and $1\leq k\leq n+1$.
\end{corollary}

\begin{proof}
Since $H^2(M; \Z)=\Z$ and $[\omega]$ is primitive integral, $c_1(M) = k[\omega]$ with $k\in\Z$. 
That $1\leq k\leq n+1$ follows from Lemmas~\ref{large} and \ref{gammaij}.
\end{proof}

\begin{example}\label{c1O}
Consider the coadjoint orbit  $\mathcal O$ of $G_2$ in Example~\ref{o}. 
Since $H^2(\mathcal O; \Z)=\Z$, we have $c_1(\mathcal O)=k[\omega]$ for some $k\in\Z$
if the symplectic class $[\omega]$ on $\mathcal O$ is primitive integral.
By 
Lemma~\ref{sub}, using the fixed points $P_0$ and $P_1$, we get
%\begin{equation}\label{c1o}
$$c_1(\mathcal O) = \frac{\Gamma_0-\Gamma_1}{\phi(P_1)-\phi(P_0)}[\omega] = 3[\omega].$$
%\end{equation}
By Lemma~\ref{gammaij}, using the fixed points $P_0$ and $P_5$, or $P_1$ and $P_3$, or $P_2$ and $P_4$, we also get
$$c_1(\mathcal O) = 3[\omega].$$
\end{example}

\begin{example}\label{c1CP5}
Consider $\CP^5$ in Example~\ref{CP5}. 
We have $H^2(\CP^5; \Z)=\Z$, so $c_1(\CP^5)=k[\omega]$ for a primitive integral symplectic class $[\omega]$ on $\CP^5$, with $k\in\Z$. The largest weight on $\CP^5$ is from $P_0$ to $P_5$ and
is equal to $\phi(P_5)-\phi(P_0)$, this number is at least 5. By Lemma~\ref{sub} or Lemma~\ref{gammaij}, we have
$$c_1(\CP^5) = 6[\omega].$$
\end{example}

\begin{example}\label{c1grass}
Consider $\Gt_2(\R^7)$ in Example~\ref{grass}.
We have $H^2(\Gt_2(\R^7); \Z)=\Z$, so $c_1(\Gt_2(\R^7))=k[\omega]$ for a primitive integral symplectic class $[\omega]$ on $\Gt_2(\R^7)$ with $k\in\Z$.
The largest weight on $\Gt_2(\R^7)$ is from $P_0$ to $P_4$ and is equal to $\phi(P_4)-\phi(P_0)$,
is also from $P_1$ to $P_5$ and is equal to $\phi(P_5)-\phi(P_1)$, this number is at least 5. By  Lemma~\ref{sub} or Lemma~\ref{gammaij}, we have
$$c_1(\Gt_2(\R^7)) = 5[\omega].$$
\end{example}

\section{on $c_1(M)$ and on the largest weight --- proof of Theorems~\ref{thm1} and \ref{thm2}}
In this section, we study the relation between $c_1(M)$ and the places where the largest weight occurs, and we prove Theorems~\ref{thm1} and \ref{thm2}.

\medskip

We start with Lemma~\ref{mult} on the multiplicities of the weights at $P_0$ and at $P_5$.

\begin{lemma}\label{mult}
Let  $(M, \omega)$ be an effective compact $10$-dimensional  Hamiltonian $S^1$-manifold with $M^{S^1}=\{P_0,  P_1, \cdots, P_5\}$. Then we have the following claims.
\begin{enumerate}
\item  Any weight at $P_0$ has maximal multiplicity $2$, and $1$ has maximal multiplicity $1$ at $P_0$. We have a similar claim for the weights at $P_5$. 
\item If $a$ and $b$ are both weights at $P_0$, $a > 1$, $a|b$, and if
$a$ or $b$ has multiplicity $2$, then the connected component $C$ of $M^{\Z_a}$ containing $P_0$ contains all the 6 fixed points, $\dim(C)=6$ when $a$ has multiplicity 2;  $1$  is not a weight at $P_0$, and $-1$ is not a weight at $P_5$.
% if $[\omega]$ is an integral class, then 
%$a |\big(\phi(P_i)-\phi(P_j)\big)$ for any $i, j$. 
We have a similar claim for the weights at $P_5$.
The triples $\pm (a, a, b)$ and $\pm (a, b, b)$ (with proper signs) cannot occur at the same time at the two extremal fixed points $P_0$ and $P_5$ .

%If $c, d, e\in\N$, and $c|d|e$, then the total multiplicity of $c, d, e$ as weights at $P_0$ is at most %3. \footnote{We will not use this last claim in this paper.}
\item  If $a$ is a weight from $P_0$ to $P_i$ for $i\geq 3$, and it is not a weight from $P_0$ to $P_i$ for $i\leq 2$, and the weights at $P_0$ and $P_5$ have symmetry, then $a$ has multiplicity $1$ at $P_0$. We have
a similar claim for the weights at $P_5$. 
\end{enumerate}
\end{lemma}

\begin{proof}
(1) If $a=1$ is a weight at $P_0$, by Lemma~\ref{alcx}, it has multiplicity $1$ since there is only one index $2$ fixed point $P_1$ in $M$. Assume $a > 1$ is a weight at $P_0$, and it has multiplicity at least $3$. Then the connected component of the $M^{\Z_a}$ containing $P_0$ has dimension at least $6$, and has at least 3 index $2$ fixed points by Lemma~\ref{alcx} applied on the component, by Morse theory and Poincar\'e duality, this component of $M^{\Z_a}$ has at least 3 coindex $2$ fixed points, and $1$ top index fixed point. This contradicts to the total number of 6 fixed points in $M$, so $a$ has maximal multiplicity $2$ at $P_0$. Similar arguments show that 
a weight at $P_0$ with multiplicity 2 is possible (in dimension 4, index 2 and coindex 2 fixed points can be the same).

(2) First assume $a$
has multiplicity $2$ at $P_0$. Then the connected component $C$ of  $M^{\Z_a}$ containing $P_0$ contains the weights $(a, a, b)$, so it has dimension at least $6$; it has at least 2 index $2$ fixed points by Lemma~\ref{alcx} used on $C$; similarly, it also has at least 2 coindex $2$ fixed points by 
Poincar\'e duality, and $1$ top index fixed point, so $C$ contains all the 6 fixed points. By Morse theory, the above fact implies that $b_2\geq 2$ and $b_{\dim(C)-2}\geq 2$. Since $C$ is symplectic, $b_{2i}(C)\neq 0$ for all $0\leq 2i\leq\dim(C)$, 
we have no more fixed points to fill in the other Betti numbers if $\dim(C) > 6$, so
 $\dim(C)=6$, and $C$ has 2 index 2 fixed points and 2 coindex 2 (or index 4) fixed points. Since $P_1$ is of index 2 in $C$, $a$ is the smallest weight on $C$ and it occurs at $P_0$ twice, by Lemma~\ref{alcx}, there is a weight $a$ from $P_0$ to $P_1$.  Since $a > 1$ is the smallest weight on $C$ and $\{P_1, P_4\}\subset C$, $1$ cannot be a weight at $P_0$ and similarly $-1$ cannot be a weight at $P_5$ by Lemma~\ref{alcx}.

Now assume $b$ has
multiplicity $2$ at $P_0$. Then the connected component $C'$ of  $M^{\Z_a}$
containing $P_0$ contains the weights $(a, b, b)$, so it has dimension at least $6$, and has at least 1 index $2$ fixed point with  weight  $-a$ by Lemma~\ref{alcx}. 
In $C'$, there is a $M^{\Z_b}$ which is of dimension at least $4$ and which contains at least 2 index $2$ fixed points each having a negative weight $-b$ by Lemma~\ref{alcx}.
So the two index 2  fixed points in $M^{\Z_b}$ are different from the index 2 fixed point in $C'$. So
$C'$ contains at least 5 fixed points. By Morse theory and Poincar\'e duality, if $C'$ cantains 5 fixed points, then $\dim(C')$ can only be possibly 8. An 8-dimensional compact Hamiltonian $S^1$-manifold with 5 fixed points has linear weights as on $\CP^4$ at the fixed points, especially, a weight at each fixed point has multiplicity 1 (see \cite{JT} and Example~\ref{CP5} to a lower dimension), a contradiction. Hence $C'$ contains all the 6 fixed points. Similar to the last case, $1$ cannot be a weight at $P_0$ and $-1$ cannot be a weight at $P_5$.

  If $(a, a, b)$ occurs at $P_0$, and $-(a, b, b)$ occurs at $P_5$, by the above, the component
$C$ of $M^{\Z_a}$ containing $P_0$ is at least 6-dimensional and contains all the fixed points, two of which are of index 2, there is a weight $a$ from $P_0$ to $P_1$, and from $P_j$ to $P_5$ since $-a$ occurs at $P_5$, where $P_j$ is of coindex 2 in $C$ ($j\neq 1$). In $C$, let the 2 index 2 fixed points be $P_1$ and $P_i$, each has a unique negative weight $-a$. Then $i\neq 1, j$ ($\index(P_j)=4$).
The component $C''$ of $M^{\Z_b}$ containing $P_5$ contains at least  4 fixed points by Lemma~\ref{alcx}, which cannot include $P_j$ and $P_0$ (since $b$ has multiplicity 1 at $P_0$). It cannot contain $P_i$ since the negative weight at a fixed point in $M^{\Z_b}$ has absolute value no smaller than $b$. So $C''$ does not have enough number of fixed points, a contradiction.
 The other claims can be proved similarly.

 (3) By the assumption and Lemma~\ref{alcx}, $a\neq 1$. Assume $a$ has multiplicity $2$ at $P_0$. By the assumption and Lemma~\ref{alcx} (used on the component), the connected component $C$ of $M^{\Z_a}$ containing $P_0$ contains at least 2 index 2 fixed points with weight $-a$, so it must contain $P_3$, $P_4$ and $P_5$. Since it contains $P_5$, by symmetry of the weights at $P_0$ and $P_5$, the component $C$ also contains $P_2$ and $P_1$, and has dimension at least 6, so $P_3$ and $P_4$ cannot both have index 2 in $C$, a contradiction.
\end{proof}

The next lemma is about the multiplicities of the largest weight at the fixed points.
\begin{lemma}\label{nm2}
Let  $(M, \omega)$ be a compact $10$-dimensional Hamiltonian $S^1$-manifold with moment map $\phi\colon M\to\R$. Assume $M^{S^1}=\{P_0,  P_1, \cdots, P_5\}$ and $[\omega]$ is a primitive integral class.  Let $c_1(M)=k[\omega]$. 
\begin{enumerate}
\item Only when $k=1$ and $k=2$, the largest weight $w$ might have multiplicity $2$ at $P_0$. If $k=2$ and $w$ has multiplicity 2 at $P_0$, then $w$ is a weight from $P_0$ to $P_1$ and $1$ is not a weight at $P_0$. We have a similar claim for $P_5$.
\item If $\phi(P_i)-\phi(P_0)=\phi(P_5)-\phi(P_{5-i})$ for $i=1, 2$ and $ k\geq 2$, then the largest weight $w$ has maximal multiplicity $1$ at $P_0$, similarly, $-w$ has maximal multiplicity 1 at $P_5$. 
\item If $k\geq 3$, the absolute value $|\pm w|$ of the largest weight $w$ has maximal multiplicity 1 at any fixed point. 
\end{enumerate}
 \end{lemma}

\begin{proof}
By Lemma~\ref{order} and Corollary~\ref{1<n+1}, 
\begin{equation}\label{eqc1}
c_1(M)=k[\omega]\,\,\,\mbox{with $k\in\N$ and $1\leq k\leq 6$}. 
\end{equation}

(1)  Assume $w$ occurs at $P_0$ with multiplicity 2  (maximal multiplicity by Lemma~\ref{mult}). 
Then $w > 1$ since $1$ occurs at most once at $P_0$ by Lemma~\ref{mult}.
Then the connected component $C$ of $M^{\Z_w}$ containing $P_0$ is (strictly) $4$ dimensional with 2 index 2 fixed points by Lemma~\ref{alcx} (applied on $C$); assume it contains $P_0$, $P_i$, $P_j$ and $P_l$ with $\phi(P_0)< \phi(P_i) < \phi(P_j) < \phi(P_l)$, where $P_i$ and $P_j$ have index 2 and $P_l$ has index 4 in $C$, and $\{i, j, l\}\subset \{1, 2, 3, 4, 5\}$.  By Lemma~\ref{gammaij}, we have 
$$k|(i+1), \,\,\, k|(j+1), \,\,\, \mbox{and} \,\,\, k|(l+2).$$ 
We see this is not possible if $k \in\{3, 4, 5, 6\}$. If $k=1$ or $k=2$, this is possible. 
If $k=2$, then it is only possible that
$i=1$, $j=3$ and $l=4$; in this case, $w$ is from $P_0$ to $P_1$, and by Lemma~\ref{alcx}, $1$ cannot be a weight at $P_0$ . 
We can prove a similar claim for $P_5$.

(2) If $\phi(P_i)-\phi(P_0)=\phi(P_5)-\phi(P_{5-i})$ for $i=1, 2$, and $k\geq 2$, then in (1), the case $k=2$ and $i=1$, $j=3$ and $l=4$ cannot happen, since there would also be a weight $w$  from $P_4$ to $P_5$ by Lemma~\ref{10}. 

(3)  Now let $k\geq 3$.  By (1) and Lemma~\ref{mult} (1),  $|\pm w|$ has maximal multiplicity 1 at $P_0$ and at $P_5$.

Assume $|\pm w|$ occurs at $P_1$. First, if $-w$ occurs at $P_1$, then $w$ is a weight from $P_0$ to $P_1$ by Lemma~\ref{10}, then $c_1(M)=2[\omega]$ by Lemma~\ref{gammaij}, contradicting to
$k\geq 3$. Now assume $w$ occurs at $P_1$. Then by Lemma~\ref{large}, similar argument yields that $w$ cannot be from $P_1$ to $P_2$. Let $P_i$ with $i\geq 3$ be the closest fixed point such that $w$ is from $P_1$ to $P_i$. If $w$ has multiplicity (at least) $2$ at $P_1$, then 
the connected component $C$ of $M^{\Z_w}$ containing $P_1$ is (at least) 4-dimensional, it contains (at least) $4$ fixed points $P_1$, $P_i$ and another 2 above $P_i$ by Lemma~\ref{alcx}, this is not possible since $C$ cannot contain $P_5$ (since $-w$ has maximal multiplicity 1 at $P_5$). This shows that $|\pm w|$ has multiplicity 1 at $P_1$.

Assume $|\pm w|$ occurs at $P_2$.  If $-w$ occurs at $P_2$ twice, then the connected component $C$ of $M^{\Z_w}$ containing $P_2$ is at least 4-dimensional, it must contain both $P_0$ and $P_1$, contradicting to $w$ having maximal multiplicity $1$ at $P_0$. If $-w$ and $w$ both occur at $P_2$,
then  the connected component $C$ of $M^{\Z_w}$ containing $P_2$ is at least 4-dimensional and contains at least a fixed point below $P_2$, contradicting to 
$|\pm w|$ having maximal multiplicity 1 at $P_0$ and at $P_1$. Similarly, we can argue that $w$ cannot occur at $P_2$ twice.

The other cases follow similarly by looking at the reversed circle action.
\end{proof}

Proposition~\ref{c1+l} claims the relation between the condition  $c_1(M)=3[\omega]$ and the places where the largest weight occurs.

\begin{proposition}\label{c1+l}
Let  $(M, \omega)$ be a compact $10$-dimensional Hamiltonian $S^1$-manifold with moment map $\phi\colon M\to\R$. Assume $M^{S^1}=\{P_0,  P_1, \cdots, P_5\}$ and $[\omega]$ is a primitive integral class. 
If $c_1(M)=3[\omega]$, then the following are the only ways the largest weight possibly occurs and at least one of them holds:
\begin{enumerate}
\item The largest weight is from $P_0$ to $P_5$ and is equal to $\frac{1}{2}\big(\phi(P_5)-\phi(P_0)\big)$.  
\item The largest weight  is from $P_0$ to $P_2$ and is equal to $\phi(P_2)-\phi(P_0)$, or is from $P_3$ to $P_5$ and is equal to $\phi(P_5)-\phi(P_3)$,  or both. 
\item The largest weight is from $P_1$ to $P_3$ and is equal to $\phi(P_3)-\phi(P_1)$, or is from $P_2$ to $P_4$ and is equal to $\phi(P_4)-\phi(P_2)$, or both.
\end{enumerate}
Conversely, if (1) holds, or (2) holds, or (3) holds and $c_1(M)=k[\omega]$ with $k>2$, 
or (3) holds and $\phi(P_i)-\phi(P_0)=\phi(P_5)-\phi(P_{5-i})$ for $i=1, 2$, then $c_1(M)=3[\omega]$, and the absolute value of the largest weight has multiplicity $1$ at the  fixed point(s) where it occurs.
\end{proposition}

\begin{proof}
If $c_1(M)=3[\omega]$, by Lemmas~\ref{gammaij} and \ref{large}, 
(1), (2), and (3) are the only possibilities where the largest weight occurs and
at least one of them holds.

Next, we prove the converse claim.  As a number, let $w\in\N$ be the largest weight.

Assume (1) holds. By Lemma~\ref{mult}, $w$ has maximal multiplicity 2 at $P_0$; similarly $-w$ has maximal multiplicity 2 at $P_5$.
By  Lemma~\ref{gammaij} used for $P_0$ and $P_5$, we obtain that $-w$ has multiplicity $1$ at $P_5$, and $c_1(M)=3[\omega]$. So $w$ also has multiplicity $1$ at $P_0$.

Assume (2) holds. If the largest weight is from $P_0$ to $P_2$ and is equal to $\phi(P_2)-\phi(P_0)$, then by Lemmas~\ref{large} and \ref{gammaij} with $s=1$, we get $c_1(M)=3[\omega]$. For the ``or" part, consider the reversed circle action, and consider similarly.  By Lemma~\ref{nm2}, $|\pm w|$ has multiplicity 1 at the fixed points where it occurs. 

Assume (3) holds and $c_1(M)=k[\omega]$ with $k>2$.  If the largest weight is from $P_1$ to $P_3$,  by Lemma~\ref{nm2}, $|\pm w|$ has multiplicity 1 at $P_1$ and $P_3$.
By Lemmas~\ref{large} and \ref{gammaij} for $P_1$ and $P_3$ with $s=1$, we get $c_1(M)=3[\omega]$.  If the largest weight is from $P_2$ to $P_4$, argue similarly.

Assume (3) holds and $\phi(P_i)-\phi(P_0)=\phi(P_5)-\phi(P_{5-i})$ for $i=1, 2$. Assume the largest weight $w$ is from $P_1$ to $P_3$.  If $-w$ occurs at $P_1$, then the $M^{\Z_w}$ containing $P_1$ is $4$-dimensional and must contain $P_0$, so it has multiplicity 2 at $P_0$. By Lemma~\ref{10}, $w=\phi(P_1)-\phi(P_0)$ is from $P_0$ to $P_1$.
By Lemma~\ref{gammaij} for $P_0$ and $P_1$, we get $c_1(M)=2[\omega]$. 
Then the multiplicity of $w$ at $P_0$ contradicts to Lemma~\ref{nm2} (2). So $-w$ does not occur at $P_1$. Then by
Lemmas~\ref{large} and \ref{gammaij} for $P_1$ and $P_3$ with $s=1$, we get $c_1(M)=3[\omega]$.  By Lemma~\ref{nm2}, $|\pm w|$ has multiplicity 1 at $P_1$ and $P_3$.
If the largest weight is from $P_2$ to $P_4$, then look at the reversed circle action and argue similarly. 
\end{proof}

To prepare to prove Theorem~\ref{thm1}, we first prove the next lemma.
\begin{lemma}\label{0235}
Let  $(M, \omega)$ be a compact $10$-dimensional Hamiltonian $S^1$-manifold with moment map 
$\phi$. Assume $M^{S^1}=\{P_0,  P_1, \cdots, P_5\}$ and $[\omega]$ is an integral class. Assume the weights $(b, b)$ occur at $P_0$ and
$(-b, -b)$ occur at $P_5$, and there are  $P_i$ and $P_j$ such that $b\nmid\big(\phi(P_i)-\phi(P_0)\big)$ and $b\nmid\big(\phi(P_5)-\phi(P_j)\big)$, where $i\neq j$ and $i, j\in\{1, 2, 3, 4\}$.
Then a connected component of $M^{\Z_b}$ containing $P_0$ is of dimension 4 and it contains 4 fixed points including $P_0$, $P_5$ and two other fixed points which are not $P_i$ and $P_j$, so there is a weight $b$ between $P_0$ and $P_5$.
\end{lemma}
\begin{proof}
Since $b\nmid\big(\phi(P_i)-\phi(P_0)\big)$ and $b\nmid\big(\phi(P_5)-\phi(P_j)\big)$, by Lemma~\ref{k|}, the connected component $C$ of $M^{\Z_b}$ containing $P_0$ does not contain $P_i$ and cannot contain $P_j$ and $P_5$ simultaneously, so $C$ can have at most 4 fixed points. By Lemma~\ref{alcx}, $C$ contains at least 2 index 2 fixed points and 1 top index fixed point, so $C$ contains at least 4 fixed points. Hence $C$ contains exactly 4 fixed points, $\dim(C)=4$, and each index 2 fixed point in $C$ has a negative weight $-b$.
Similarly, the connected component $C'$ of $M^{\Z_b}$ containing $P_5$ has dimension 4 and 
contains 4 fixed points, each of its coindex 2 fixed point has a positive weight $b$.
Since $C$ and $C'$ must have a non-extremal fixed point in common, $C=C'$, so it contains $P_0$, $P_5$ and the other 2 non-extremal fixed points other than $P_i$ and $P_j$.  
\end{proof}

We can prove Theorem~\ref{thm1} now in the Introduction.

\begin{theorem 1}
Let  $(M, \omega)$ be a compact $10$-dimensional Hamiltonian $S^1$-manifold. Assume $M^{S^1}=\{P_0,  P_1, \cdots, P_5\}$ and $[\omega]$ is a primitive integral class. Then $c_1(M)=k[\omega]$ with $k\in\N$ and $1\leq k\leq 6$, and the largest weight is at least 5.
\end{theorem 1}

\begin{proof}
Lemma~\ref{order} and Corollary~\ref{1<n+1} give the first claim. 

Next, we prove the second claim.
Let $\phi\colon M\to \R$ be the moment map.
Note that there are 5 weights at $P_0$.
By Lemma~\ref{mult} (1), $1$ has multiplicity at most 1 at $P_0$, so 
$1$ cannot be the largest weight on $M$; similarly, since $2$ has maximal multiplicity $2$ at $P_0$, 
the set of weights cannot be $\{1, 2, 2\}$ at $P_0$, so $2$ cannot be the largest weight on $M$. 

Assume $3$ is the largest weight on $M$.  By Lemmas~\ref{mult} (1) and Lemma~\ref{nm2} (1), 
if $k\geq 2$, then the set of weights at $P_0$ can only be $\{3, 3, 2, 2\}$ or $\{3, 2, 2, 1\}$, not enough number of weights, a contradiction. 
If $k=1$, besides the above 2 impossible cases, the set of weights at $P_0$ might be $\{3, 3, 2, 2, 1\}$; similarly, the set of weights at $P_5$ might be  $- \{3, 3, 2, 2, 1\}$. 
By Lemmas~\ref{alcx} and \ref{10}, $1=\phi(P_1)-\phi(P_0)=\phi(P_5)-\phi(P_4)$.
By Lemma~\ref{0235}, $3$ is a  weight between $P_0$ and $P_5$. Then the sets of weights at $P_0$ and $P_5$  contradict to Lemma~\ref{mod}.  We have proved that $3$ cannot be the largest weight on $M$; and the match $\{3, 3, 2, 2, 1\}$ and $- \{3, 3, 2, 2, 1\}$ cannot be respectively the sets of weights at $P_0$ and $P_5$.

Now assume $4$ is the largest weight on $M$. By Lemmas~\ref{mult} (2)
and \ref{nm2}, the possible set of weights at $P_0$ 
is $\{4, 4, 2, 3, 3\}$ or $\{4, 4, 1, 3, 3\}$ 
or  $\{4, 1, 2, 3, 3\}$ or $\{4, 2, 2, 3, 3\}$ or $\{1, 2, 2, 3, 3\}$. Similarly, the set of weights at $P_5$ is the negative of one of these sets. 

If the set of weights at $P_0$ is $\{4, 4, 2, 3, 3\}$, then by Lemma~\ref{mult} (2) and the possibilities above on the set of weights at $P_5$,  the set of weights at $P_5$ is $-\{4, 4, 2, 3, 3\}$. 
By Lemmas~\ref{mult} (2) and \ref{10},  $2=\phi(P_1)-\phi(P_0)=\phi(P_5)-\phi(P_4)$. By Lemma~\ref{0235}, there is a weight $3$ (or $4$) between $P_0$ and $P_5$, the sets of weights at $P_0$ and $P_5$ contradict to Lemma~\ref{mod}. Hence $\{4, 4, 2, 3, 3\}$ cannot be the set of weights at $P_0$. Similarly, $-\{4, 4, 2, 3, 3\}$  cannot be the set of weights at $P_5$.

Similar to the last case, $\{4, 2, 2, 3, 3\}$ cannot be the set of weights at $P_0$.  
Similarly $-\{4, 2, 2, 3, 3\}$ cannot be the set of weights at $P_5$.

If  $\{4, 4, 1, 3, 3\}$ is the set of weights at $P_0$,  by the possibilities on the set of weights at $P_5$ and what we have excluded,  the set of weights at $P_5$ is  $-\{4, 4, 1, 3, 3\}$ or $-\{4, 1, 2, 3, 3\}$ 
or $-\{1, 2, 2, 3, 3\}$. By Lemmas~\ref{alcx} and \ref{10}, 
$1=\phi(P_1)-\phi(P_0)=\phi(P_5)-\phi(P_4)$. By Lemma~\ref{0235}, $3$ is a weight between $P_0$ and $P_5$. The matches all contradict to Lemma~\ref{mod}. Hence $\{4, 4, 1, 3, 3\}$ cannot be the set of weights at $P_0$. Similarly,  $-\{4, 4, 1, 3, 3\}$ cannot be the set of weights at $P_5$.

If $\{4, 1, 2, 3, 3\}$ is the set of weights at $P_0$, by the possibilities on the set of weights at $P_5$ and what we have excluded, the set of weights at $P_5$ is $-\{4, 1, 2, 3, 3\}$ or $-\{1, 2, 2, 3, 3\}$. We can similarly as above exclude the possibility that $\{4, 1, 2, 3, 3\}$ and $-\{4, 1, 2, 3, 3\}$ are respectively the sets of weights at $P_0$ and at $P_5$. Now assume $-\{1, 2, 2, 3, 3\}$ is the set of weights at $P_5$. As above, we have $1=\phi(P_1)-\phi(P_0)=\phi(P_5)-\phi(P_4)$.
By Lemma~\ref{0235}, a component of $M^{\Z_3}$ containing $P_0$ contains $P_0$, $P_2$, $P_3$ and $P_5$. We can similarly argue that a component of $M^{\Z_2}$ containing $P_5$ contains 
$P_0$, $P_2$, $P_3$ and $P_5$. So the largest weight 4 is from $P_0$ to $P_2$ or from $P_0$ to $P_3$.  If $4$ is from $P_0$ to $P_3$, then by Lemma~\ref{k|}, $3$ and $4$  both divide $\phi(P_3)-\phi(P_0)$, together with Lemma~\ref{gammaij} (with $s=1$), we get $\frac{\phi(P_3)-\phi(P_0)}{4}$ divides 4, a contradiction. Now assume $4$ is from $P_0$ to $P_2$. Then by Lemma~\ref{k|}, $3$ and $4$ both divide $\phi(P_2)-\phi(P_0)$, together with Lemma~\ref{gammaij} (with $s=1$), we get $\phi(P_2)-\phi(P_0)=12$ and $c_1(M)=[\omega]$.
Then by Lemma~\ref{sub} for $P_0$ and $P_5$, and the fact that $\phi(P_3)-\phi(P_2)$ and
$\phi(P_5)-\phi(P_3)$ are divisible by $2$ and $3$, we get $\phi(P_3)-\phi(P_2)=6$, and 
$\phi(P_5)-\phi(P_3)=6$. We already know that $\{2, 3\}$ is the set of positive weights at $P_3$,
and $-2$ and $-3$ are negative weights at $P_3$. By Lemma~\ref{<} (and the fact that $4$ is the largest weight), we see that the other negative weight at $P_3$ has absolute value no bigger than $3$. Hence $3$ is the largest weight at $P_3$ and at $P_5$ and $3$ is between $P_3$ and $P_5$.
Look at the manifold upside down and using Lemma~\ref{gammaij} for $P_3$ and $P_5$, we see a contradiction that $3=1\cdot 2$.
We have shown that $\{4, 1, 2, 3, 3\}$ cannot be the set of weights at $P_0$. Similarly, $-\{4, 1, 2, 3, 3\}$ cannot be the set of weights at $P_5$.

We already excluded the possibility that $\pm\{1, 2, 2, 3, 3\}$ are the sets of weights at $P_0$ and $P_5$.

The argument above shows that $4$ cannot be the largest weight on $M$.

By Examples~\ref{o}, \ref{CP5}, and \ref{grass}, the largest weight can be 5.
\end{proof}

Based on Theorem~\ref{thm1}, in the rest of this section, we concentrate on the case that the largest weight is 5.

Assume $c_1(M)=3[\omega]$, and the largest weight is $5$, the next lemma  excludes the possibility that the largest weight is from $P_0$ to $P_2$, or from $P_3$ to $P_5$, i.e., the possible case (2) in Proposition~\ref{c1+l}.

\begin{lemma}\label{n02}
Let  $(M, \omega)$ be a compact $10$-dimensional Hamiltonian $S^1$-manifold with moment map $\phi\colon M\to\R$. Assume $M^{S^1}=\{P_0,  P_1, \cdots, P_5\}$ and $[\omega]$ is a primitive integral class. 
If $c_1(M)=3[\omega]$ and the largest weight is $5$, 
then $5$  cannot be from $P_0$ to $P_2$ and is equal to $\phi(P_2)-\phi(P_0)$, similarly, it cannot be from $P_3$ to $P_5$ and is equal to $\phi(P_5)-\phi(P_3)$.
 \end{lemma}

\begin{proof}
Since $c_1(M)=3[\omega]$, by Proposition~\ref{c1+l}, 
the largest weight $5$  could be from $P_0$ to $P_2$ and is equal to $\phi(P_2)-\phi(P_0)$, 
or from $P_3$ to $P_5$ and is equal to $\phi(P_5)-\phi(P_3)$, or both. 
With no loss of generality, we assume that the largest weight $5$ is from $P_0$ to $P_2$ and 
\begin{equation}\label{eq02}
5=\phi(P_2)-\phi(P_0).
\end{equation} 
We will argue that this is impossible.
By Proposition~\ref{c1+l}, $5$ has multiplicity 1 at $P_0$.
We notice a few facts:
\begin{enumerate}
\item The triples $\pm (2, 2, 4)$ and $\pm (2, 4, 4)$ cannot occur at $P_0$ nor at $P_5$ (otherwise by Lemmas~\ref{mult} (2) and \ref{k|}, $2\,|\,\big(\phi(P_2)-\phi(P_0)\big)$, contradicting to (\ref{eq02}).
\item  If $-5$ is a weight at $P_5$, then $-5$ is between $P_3$ and $P_5$ and $5=\phi(P_5)-\phi(P_3)$ (since $c_1(M)=3[\omega]$ and $5$ has multiplicity 1 at $P_0$), and $-5$ has multiplicity 1 at $P_5$ (by Proposition~\ref{c1+l}). 
%\item If $-1$ is a weight at $P_5$, then it is of multiplicity 1, is a weight between $P_4$ and $P_5$ %and $1=\phi(P_5)-\phi(P_4)$ (by Lemmas~\ref{alcx} and \ref{10}).  
\item At least one of $-5$ and $-1$ is a weight at $P_5$ (otherwise $- (2, 2, 4)$ or $- (2, 4, 4)$ would occur). 
\end{enumerate}
First assume both $5$ and $1$ occur at $P_0$. 
Then $1=\phi(P_1)-\phi(P_0)$ by Lemmas~\ref{alcx} and \ref{10}, and $1$ has multiplicity 1 at $P_0$ by Lemma~\ref{mult} (1).
If $2$ occurs at $P_0$ with multiplicity 2, then the $M^{\Z_2}$ containing $P_0$ contains at least 2 index 2 fixed points by Lemma~\ref{alcx}, and it cannot contain $P_1$ and $P_2$ by Lemma~\ref{k|}, so it contains $P_3$, $P_4$ and $P_5$, then $2$ divides the difference of the moment map values of any two of them. Then neither $-5$ nor $-1$ can be a weight at $P_5$ by (2), Lemmas~\ref{alcx} and \ref{10}, contradicting to (3). We can similarly get that $3$ and $4$ has maximal multiplicity 1 at $P_0$. Hence in this case, $\{5, 1, 2, 3, 4\}$ is the possible set of weights at $P_0$. Now, assume $1$ does not occur at $P_0$, then by Lemmas~\ref{alcx} and \ref{10}, $\phi(P_1)-\phi(P_0)$ is $2$ or $3$ or $4$ and is a weight at $P_0$.  Consider these possibilities and consider similarly as above, we see that the possible set of weights at $P_0$ is $\{5, 3, 2, 3, 4\}$, in which case $\phi(P_1)-\phi(P_0)=3$.
So the possible set of weights and the sum of weights at $P_0$ are
\begin{itemize}
\item [(1$P_0$)] $\{5, 1, 2, 3, 4\}$ and $\Gamma_0 = 15$, or
\item [(2$P_0$)] $\{5, 3, 2, 3, 4\}$ and $\Gamma_0 = 17$.
\end{itemize}
Similarly, if $-5$ and $-1$ both occur at $P_5$, then $-2$, $-3$ and $-4$ each has maximal multiplicity 1 at $P_5$. Together with the restrictions above, we get the possible set of weights and the sum of weights at $P_5$:
\begin{itemize}
\item [(1$P_5$)] $\{-5, -1, -2, -3, -4\}$ and $\Gamma_5 = -15$, or
\item [(2$P_5$)] $\{-5, -3, -2, -3, -4\}$ and $\Gamma_5 = -17$, or
\item [(3$P_5$)] $\{-1, -3, -2, -3, -2\}$  and $\Gamma_5 = -11$, or
\item [(4$P_5$)] $\{-1, -3, -2, -3, -4\}$  and $\Gamma_5 = -13$, or
\item [(5$P_5$)] $\{-1, -3, -4, -3, -4\}$  and $\Gamma_5 = -15$.
\end{itemize}
By Lemma~\ref{sub}, 
\begin{equation}\label{0-5}
\Gamma_0 -\Gamma_5 = 3 \big(\phi(P_5)-\phi(P_0)\big).
\end{equation} 
Or $3\,|\, \big(\Gamma_0 -\Gamma_5\big)$.   So we only need to consider
the matches (1$P_0$) and (1$P_5$),  (1$P_0$) and (5$P_5$), and  (2$P_0$) and (4$P_5$).
First, (1$P_0$) and (1$P_5$) cannot occur since $\phi(P_5)-\phi(P_0) > 10$,  contradicting to  (\ref{0-5}).
Secondly, (1$P_0$) and (5$P_5$) cannot occur, since by Lemma~\ref{alcx} the $M^{\Z_3}$ containing $P_5$ contains at least 4 fixed points, but by Lemma~\ref{k|} it cannot contain $P_4$, and $P_0$ ($3$ has multiplicity 1 at $P_0$), and it cannot contain $P_1$ and $P_2$ simultaneously (since $\phi(P_2)-\phi(P_1) = 4$), not enough number of fixed points. Finally, for the case (2$P_0$) and (4$P_5$), 
we have (\ref{eq02}), and $1=\phi(P_5)-\phi(P_4)$ by Lemmas~\ref{alcx} and \ref{10}, 
by Lemma~\ref{0235},  a connected $M^{\Z_3}$ containing $P_0$ contains 4 fixed points including $P_5$,  so
$$\phi(P_5)-\phi(P_0) = 3k \geq 9,$$ 
contradicting to (\ref{0-5}).
\end{proof}

Next, for the case that the largest weight is $5$, we study the relation between the possible cases
(1) and (3) in Proposition~\ref{c1+l} and their relations with the condition $c_1(M) = 3[\omega]$.

In the next three lemmas, we start from condition (1) in Proposition~\ref{c1+l}.

\begin{lemma}\label{w0}
Let  $(M, \omega)$ be a compact $10$-dimensional Hamiltonian $S^1$-manifold with moment map $\phi\colon M\to\R$. Assume $M^{S^1}=\{P_0,  P_1, \cdots, P_5\}$ and $[\omega]$ is a primitive integral class. Assume the largest weight  is from $P_0$ to $P_5$ and is equal to $\frac{1}{2}\big(\phi(P_5)-\phi(P_0)\big)=5$. Then each of $1$, $2$, $3$, $4$ and $5$ has maximal multiplicity 1 at $P_0$. Similarly, each of $-1$, $-2$, $-3$, $-4$ and $-5$ has maximal multiplicity 1 at $P_5$.
 \end{lemma}

\begin{proof}
 By Proposition~\ref{c1+l}, 
\begin{equation}\label{c1-05}
c_1(M)=3[\omega],
\end{equation}
the largest weight $5$ has multiplicity 1 at $P_0$ and $-5$ has multiplicity 1 at $P_5$.

Since the largest weight is $5$, the possible distinct weights at $P_0$ and $P_5$ are $\pm 1$, $\pm 2$, $\pm 3$, $\pm 4$ and $\pm 5$. 
By Lemma~\ref{mod}, the possible relations for the weights at $P_0$ and $P_5$ are among the following ones:
\begin{equation}\label{mod5}
1=-4, \,\,\, 2=-3, \,\,\, 3=-2, \,\,\,4=-1,    \mod 5.
\end{equation}
Since by Lemma~\ref{alcx}, $1$ has maximal multiplicity 1 at $P_0$ and $-1$ has maximal multiplicity 1 at $P_5$, by (\ref{mod5}), $4$ has maximal multiplicity 1 at $P_0$ and $-4$ has maximal multiplicity 1 at $P_5$. Next we show that $3$ has maximal multiplicity 1 at $P_0$. Assume $3$ has multiplicity 2 at $P_0$ (maximal multiplicity by Lemma~\ref{mult} (1)). We have 3 weights now $5, 3, 3$ at $P_0$,  and 3 weights  $-5, -2, -2$ at $P_5$ by (\ref{mod5}). If $1$ occurs at $P_0$, then $-4$ occurs at $P_5$ by (\ref{mod5}), then by Lemma~\ref{mult} (2), $1$ cannot be a weight at $P_0$, a contradiction. Hence $1$ cannot occur at $P_0$. Then we have 2 possibilities for the sets of weights at $P_0$ and $P_5$:
\begin{itemize}
\item [(i)] $P_0\colon \{5, 3, 3,  2, 2\}$, and  $P_5\colon \{-5, -2, -2, -3, -3\}$.
\item [(ii)] $P_0\colon \{5, 3, 3,  2, 4\}$, and $P_5\colon \{-5, -2, -2, -3, -1\}$.
\end{itemize}
  By the assumption, 
$$\phi(P_5)-\phi(P_0)=10.$$
Consider (i).  By Lemma~\ref{10}, $\phi(P_5)-\phi(P_4)\geq 2$, so $\phi(P_4)-\phi(P_0)\leq 8$. 
Since the $M^{\Z_3}$ containing 
$P_0$ contains  4 fixed points by Lemma~\ref{alcx}, and  the difference of the moment map values of any two of them is divisible by $3$ by Lemma~\ref{k|}, this is not possible.
Now consider  (ii). By Lemmas~\ref{alcx} and \ref{10}, we have 
$\phi(P_5)-\phi(P_4) = 1$, so 
$$\phi(P_4)-\phi(P_0) = 9.$$ 
For the same reason as in  (i), the top index 4 fixed point in $M^{\Z_3}$ containing $P_0$ must be $P_4$. If this $M^{\Z_3}$ does not contain $P_1$, then it contains
$P_0$, $P_2$, $P_3$ and $P_4$, and then none of $\phi(P_i)-\phi(P_0)$ ($i\neq 0$) is divisible by $4$, contradicting to $4$ being a weight at $P_0$. So this $M^{\Z_3}$ contains $P_1$ and 
$$\phi(P_1)-\phi(P_0)=3.$$ 
The $M^{\Z_2}$ containing $P_5$ is 4-dimensional, it contains 2 coindex 2 fixed points by Lemma~\ref{alcx}, and it cannot contain $P_4$ and $P_1$ since the difference of the moment map values of any two fixed points in $M^{\Z_2}$ is divisible by 2 by Lemma~\ref{k|}. So $M^{\Z_2}$ contains $P_0$, $P_2$, $P_3$ and $P_5$. If the $M^{\Z_3}$ above contains $P_0$, $P_1$, $P_2$ and $P_4$ ($\phi(P_4)-\phi(P_2)=3$),
then for the $M^{\Z_2}$ to contain $P_3$, we have $\phi(P_4)-\phi(P_3)=1$ and 
$\phi(P_3)-\phi(P_2)=2$. Then none of $\phi(P_5)-\phi(P_i)$ with $i< 5$ is divisible by 3, contradicting to $-3$ being a weight at $P_5$. Now assume the $M^{\Z_3}$ contains $P_0$, $P_1$, $P_3$ and $P_4$. Then $\phi(P_4)-\phi(P_3)=\phi(P_3)-\phi(P_1)=3$.
For the $M^{\Z_2}$ to contain $P_2$, we must have $\phi(P_3)-\phi(P_2)=2$
and $\phi(P_2)-\phi(P_1)=1$. Note that the weight $-3$ at $P_5$ is between $P_5$ and $P_2$ since it only divides $\phi(P_5)-\phi(P_2)$, and the weight 4 at $P_0$ is between $P_0$ and $P_2$ for the same reason. By Lemmas~\ref{<} and \ref{other}, the other negative weight at $P_2$ (except $-4$) can only be $-1$. We know 2 positive weights at $P_2$, $2$ and $3$. Note that $\phi(P_4)-\phi(P_2)=5$. If there is a weight $5$ between $P_2$ and $P_4$, then the set of weights at $P_2$ is
$\{-4, -1, 2, 3, 5\}$, and we (already) know 4 weights at $P_4$: $-3$, $-3$, $-5$, $1$. 
We do not have enough weights at $P_2$ which equal to the two $-3$ at $P_4$ mod $5$, 
contradicting to Lemma~\ref{mod}.  Hence $-5$ cannot be a weight at $P_4$. Using Lemma~\ref{<} and what we already know, we see that the largest weight between $P_4$ and $P_3$ is $-3$ and
$3$ is the largest among all the weights at $P_4$ and $P_3$. Using Lemma~\ref{gammaij} for $-\phi$ (or for the manifold upside down) and for $P_3$ and $P_4$, we see that $c_1(M)=2[\omega]$, contradicting to (\ref{c1-05}). So
(ii) is not possible.
We have shown that $3$ has maximal multiplicity 1 at $P_0$. Similarly, $-3$ has maximal multiplicity 1 at $P_5$. Hence $2$ and $-2$ respectively have maximal multiplicity 1 at $P_0$ and $P_5$ by (\ref{mod5}).
\end{proof}

\begin{lemma}\label{w}
Let  $(M, \omega)$ be a compact $10$-dimensional Hamiltonian $S^1$-manifold with moment map $\phi\colon M\to\R$. Assume $M^{S^1}=\{P_0,  P_1, \cdots, P_5\}$ and $[\omega]$ is a primitive integral class. Assume the largest weight is from $P_0$ to $P_5$ and is equal to $\frac{1}{2}\big(\phi(P_5)-\phi(P_0)\big)=5$. Then the following are true.
\begin{enumerate}
\item The set of weights at $P_0$ is $\{5, 4, 3, 2, 1\}$ and the set of weights at $P_5$ is 
$-\{5, 4, 3, 2, 1\}$.
\item $\phi(P_1)-\phi(P_0)=\phi(P_5)-\phi(P_4) = 1$, $\phi(P_2)-\phi(P_1)=\phi(P_4)-\phi(P_3)=3$, and $\phi(P_3)-\phi(P_2)=2$.
\end{enumerate}
\end{lemma}

\begin{proof}
(1) By Lemma~\ref{w0},  each of $1$, $2$, $3$, $4$ and $5$ has maximal multiplicity one at $P_0$, so the set of weights at $P_0$ is $\{1, 2, 3, 4, 5\}$.  Similarly, the set of weights at $P_5$ is $-\{1, 2, 3, 4, 5\}$.

(2) By symmetry of the weights at $P_0$ and $P_5$, we have that
$$\phi(P_5)-\phi(P_4)=\phi(P_1)-\phi(P_0)\quad\mbox{and}\quad \phi(P_4)-\phi(P_3)=\phi(P_2)-\phi(P_1).$$
By (\ref{c1-05}) and Lemma~\ref{sub} for $P_0$ and $P_5$, we get that $\phi(P_3)-\phi(P_2) \geq 2$ is even.
Then we have the following possibilities:
\begin{itemize}
\item [(a)] $\phi(P_1)-\phi(P_0)=\phi(P_5)-\phi(P_4) = 1$, $\phi(P_2)-\phi(P_1)=\phi(P_4)-\phi(P_3)=3$, $\phi(P_3)-\phi(P_2)=2$.
\item [(b)]  $\phi(P_1)-\phi(P_0)=\phi(P_5)-\phi(P_4) = 1$, $\phi(P_2)-\phi(P_1)=\phi(P_4)-\phi(P_3)=2$, $\phi(P_3)-\phi(P_2)=4$.
\item [(c)]  $\phi(P_1)-\phi(P_0)=\phi(P_5)-\phi(P_4) = 1$, $\phi(P_2)-\phi(P_1)=\phi(P_4)-\phi(P_3)=1$, $\phi(P_3)-\phi(P_2)=6$.
\end{itemize}
Case (a) exists by Example~\ref{o}. Case (b) does not exist since the weight $4$ at $P_0$ divides none of
$\phi(P_0)-\phi(P_i)$ for $i\neq 0$.
Now consider (c). The values $\phi(P_i)-\phi(P_0)$ for $i > 0$ are $1$, $2$, $8$, $9$ and $10$. Note that the weight $3$ at $P_0$ can only be between $P_0$ and $P_4$. 
The positive weight at $P_4$ is $1$ by Lemma~\ref{10}.  By Lemma~\ref{<} and the fact that $5$ is the largest weight, except $-3$, the other possible distinct negative weights at $P_4$ are $-1$, $-2$, and $-4$, where $-2$ and $-4$ can only be between $P_4$ and $P_1$, each has maximal multiplicity 1 and cannot occur at the same time (otherwise a connected component of $M^{\Z_2}$
or $M^{\Z_4}$ is of  dimension at least 4 but with 2 fixed points $P_1$ and $P_4$, noting that the component cannot contain other fixed points by Lemma~\ref{k|}). So the set of weights at $P_4$ 
 is possibly as follows.
\begin{itemize}
\item [(c1)] $\{1, -1, -1, -2, -3\}$.
\item [(c2)] $\{1, -1, -1, -4, -3\}$.
\end{itemize}
By Lemma~\ref{sub} for $P_0$ and $P_4$, we get that  (c1) and (c2) contradict to (\ref{c1-05}), so (c) is not possible.
\end{proof}

\begin{remark}\label{remw}
In Lemma~\ref{w}, if we only assume that $\frac{1}{2}\big(\phi(P_5)-\phi(P_0)\big)=5$ 
is a weight between $P_0$ and $P_5$, and it is the largest among all the weights at $P_0$ and $P_5$ (instead of that it is the largest weight on $M$), then we cannot exclude Case (c). We still have that the set of weights at $P_0$ and $P_5$:
$$P_0\colon \{1, 2, 3, 4, 5\}, \,\,\, P_5\colon -\{1, 2, 3, 4, 5\}.$$
But (in (c)) we may have that the set of weights at $P_4$ is
$$P_4\colon \{1, -1, -2, -3, -7\},$$
where $7$ is a weight between $P_2$ and $P_4$. Using this, we can get the set of weights at $P_2$:
$$P_2\colon \{-1, -2, 1, 7, 4\}.$$
 By symmetry, the set of weights at $P_1$ is the negative of the set of weights at $P_4$, 
and the set of weights at $P_3$ is the negative of the set of weights at $P_2$. So we may get  the weights at all the fixed points. By Lemma~\ref{gammaij} or \ref{sub}, we have
$$c_1(M)=3[\omega].$$
By Lemma~\ref{ring} in the next section, we can get that the ring $H^*(M; \Z)$ is generated by the elements:
$$1, \,\,\, [\omega], \,\,\, [\omega]^2, \,\,\, \frac{1}{12}[\omega]^3,\,\,\,\frac{1}{12}[\omega]^4, \,\,\,\frac{1}{12}[\omega]^5.$$
\end{remark}

\begin{remark}\label{imsym}
By Lemma~\ref{w}, the condition that $\frac{1}{2}\big(\phi(P_5)-\phi(P_0)\big) = 5$  is the largest weight and is a weight from $P_0$ to $P_5$ implies symmetry of the image of $\phi$ and symmetry of the set of weights at $P_0$ and $P_5$.
\end{remark}

\begin{lemma}\label{w'}
Let  $(M, \omega)$ be a compact $10$-dimensional Hamiltonian $S^1$-manifold with moment map $\phi\colon M\to\R$. Assume $M^{S^1} = \{P_0,  P_1, \cdots, P_5\}$ and $[\omega]$ is a primitive integral class.  If the largest weight is from $P_0$ to $P_5$ and is equal to $5= \frac{1}{2}\big(\phi(P_5)-\phi(P_0)\big)$, then the set of weights at $P_1$ is
$$P_1\colon \{-1, 1, 5, 4, 3\},$$
where $5 =\phi(P_3)-\phi(P_1)$ is a weight from $P_1$ to $P_3$, and the set of weights at $P_4$ is 
$$P_4\colon - \{3, 4, 5, 1, -1\},$$
where $5 =\phi(P_4)-\phi(P_2)$ is a weight from $P_2$ to $P_4$.
\end{lemma}

\begin{proof}
By Lemma~\ref{w}, $1$ is a weight at $P_0$.
By Lemma~\ref{alcx},  $-1$ is the negative weight at $P_1$.
By Lemma~\ref{w} (2), the values of $\phi(P_i)-\phi(P_1)$ for $i > 1$ are $3$, $5$, $8$ and $9$. By Lemma~\ref{<} and the fact that $5$ is the largest weight on $M$, 
we get that the possible distinct positive weights at $P_1$ are $5$, $4$, $3$, $2$ and $1$. Note that $5$ can only divide $\phi(P_3)-\phi(P_1)$, so it has maximal multiplicity $1$. Note that $4$ and $2$ can only divide $\phi(P_4)-\phi(P_1)$, if any of them has multiplicity more than $1$ or they both occur at $P_1$, then the connected component of $M^{\Z_2}$ or $M^{\Z_4}$ containing $P_1$ only contains $P_1$ and $P_4$ and has dimension bigger than $2$, a contradiction. So
$4$ and $2$ each has maximal multiplicity $1$ and cannot occur at the same time at $P_1$. Note that  $3$ can only divide $\phi(P_2)-\phi(P_1)$ and $\phi(P_5)-\phi(P_1)$, by Lemma~\ref{alcx} used on the component of $M^{\Z_3}$ containing $P_1$, we see that $3$ has maximal multiplicity $1$ at $P_1$. If $1$ occurs at $P_1$ with multiplicity $2$, then $-1$ occurs at $P_2$ with multiplicity $2$ by Lemma~\ref{alcx}. Note that the weight $4$ at $P_0$ only divides $\phi(P_2)-\phi(P_0)$, so $4$ is a weight from $P_0$ to $P_2$; so $-1$ cannot occur at $P_2$ twice.
Then $1$ has maximal multiplicity $1$ at $P_1$. These facts together yield that the set of weights at $P_1$ is $\{-1, 5, 4, 3, 1\}$ or $\{-1, 5, 3, 2, 1\}$. By the symmetry in Lemma~\ref{w} (2), the set of weights at $P_4$ is  $-\{-1, 5, 4, 3, 1\}$ or $-\{-1, 5, 3, 2, 1\}$. Using Lemma~\ref{sub} for $P_1$ and $P_4$, we get that the case of $\pm \{-1, 5, 3, 2, 1\}$ contradicts to (\ref{c1-05}). By Example~\ref{o}, the case of  $\pm\{-1, 5, 4, 3, 1\}$ exists. 

From the above, the weight $5$ at $P_1$ can only be from $P_1$ to $P_3$ and $5=\phi(P_3)-\phi(P_1)$. The rest is by symmetry.
\end{proof}

Next, we start from condition (3) in Proposition~\ref{c1+l}.

\begin{lemma}\label{lw1}
Let  $(M, \omega)$ be a compact $10$-dimensional Hamiltonian $S^1$-manifold with moment map $\phi\colon M\to\R$. Assume $M^{S^1}=\{P_0,  P_1, \cdots, P_5\}$ and $[\omega]$ is a primitive integral class. Assume the largest weight is from $P_1$ to $P_3$  and is equal to $5=\phi(P_3)-\phi(P_1)$, or is from $P_2$ to $P_4$ and is equal to $5 = \phi(P_4)-\phi(P_2)$, or both. 
Assume $c_1(M)=k[\omega]$ with $2 < k\in\N$ or $\phi(P_i)-\phi(P_0)=\phi(P_5)-\phi(P_{5-i})$ for $i=1, 2$. Then the following hold:
\begin{enumerate}
\item the largest weight is from $P_0$ to $P_5$ and is equal to
$5=\frac{1}{2}\big(\phi(P_5)-\phi(P_0)\big)$.
\item $\phi(P_1)-\phi(P_0)=\phi(P_5)-\phi(P_4)=1$,
$\phi(P_3)-\phi(P_2)=2$, $\phi(P_2)-\phi(P_1)=\phi(P_4)-\phi(P_3)=3$.
\end{enumerate}
\end{lemma}

\begin{proof}
By Proposition~\ref{c1+l},
\begin{equation}\label{eq13}  
c_1(M)=3[\omega],
\end{equation}
and $|\pm 5|$ has multiplicity 1 at the corresponding fixed point. First assume the largest weight $5$ occurs  at $P_0$. By Lemma~\ref{n02}, $5$ is not from $P_0$ to $P_2$ (nor from $P_3$ to $P_5$).  Then by  Lemma~\ref{gammaij}, $5$ is from $P_0$ to $P_5$ and $5=\frac{1}{2}\big(\phi(P_5)-\phi(P_0)\big)$. 
By Lemma~\ref{w}, (2) also holds. By Example~\ref{o}, this case exists.

Next assume the largest weight $5$ does not occur at $P_0$, then it also does not occur at $P_5$ (otherwise  it would occur at $P_0$ by a similar argument above). We will show that this is impossible. Now the possible distinct weights at $P_0$ are $1, 2, 3, 4$. 
Note that $2$ does not divide  $5=\phi(P_3)-\phi(P_1)$ or $5=\phi(P_4)-\phi(P_2)$, then by Lemma~\ref{mult} (2) and Lemma~\ref{k|}, $\pm (2, 2, 4)$ or $\pm (2, 4, 4)$ cannot occur at $P_0$ and $P_5$. To have 5 weights at $P_0$, $1$ must occur at $P_0$. Similarly, $-1$ needs to occur at $P_5$.
By Lemma~\ref{mult} (1), they have multiplicity $1$. By Lemmas~\ref{alcx} and \ref{10},
$$\phi(P_1)-\phi(P_0) = \phi(P_5)-\phi(P_4) =1.$$
So the set of weights at $P_0$ is $\{1, 2, 2, 3, 3\}$ or $\{1, 2, 4, 3, 3\}$ or $\{1, 4, 4, 3, 3\}$ (by Lemma~\ref{mult} (1)). Similarly, the set of weights at $P_5$ is  $-\{1, 2, 2, 3, 3\}$ or $-\{1, 2, 4, 3, 3\}$ or $-\{1, 4, 4, 3, 3\}$. Because of (\ref{eq13}), by Lemma~\ref{sub}, we have $3\,| \big(\Gamma_0-\Gamma_5\big)$. Hence we may only possibly have the following matches of the sets of weights at $P_0$ and $P_5$:
\begin{itemize}
\item [(i)] $P_0\colon \{1, 2, 2, 3, 3\}$ and $P_5\colon -\{1, 2, 4, 3, 3\}$ or vice versa with signs exchanged, or
\item [(ii)] $P_0\colon \{1, 4, 4, 3, 3\}$ and $P_5\colon -\{1, 4, 4, 3, 3\}$.
\end{itemize}
By the assumption, without loss of generality, we assume that 
\begin{equation}\label{513}
5= \phi(P_3)-\phi(P_1).
\end{equation}
 Consider (i).  By Lemma~\ref{0235}, the component of $M^{\Z_3}$ containing $P_0$ also contains $P_2$, $P_3$ and $P_5$. So $\phi(P_3)-\phi(P_2)$ is divisible by $3$; 
because of (\ref{513}), we have 
$$\phi(P_3)-\phi(P_2)=3.$$ 
Then by Lemma~\ref{k|}, the component of $M^{\Z_2}$ containing $P_0$ cannot contain $P_1$, cannot contain $P_4$ and $P_5$ simultaneously, and cannot contain $P_2$ and $P_3$  simultaneously, while it has at least 4 fixed points by Lemma~\ref{alcx}, a contradiction. Consider (ii). Consider as in  (i), or by Lemma~\ref{0235}, there is a weight $3$ between $P_0$ and $P_5$, but the sets of weights at $P_0$ and $P_5$ contradict to Lemma~\ref{mod}.
\end{proof}

Now, we can prove our main  Theorem 2 in the Introduction.

\begin{theorem 2}
Let  $(M, \omega)$ be a  compact $10$-dimensional Hamiltonian $S^1$-manifold with moment map $\phi\colon M\to\R$. Assume $M^{S^1}=\{P_0,  P_1, \cdots, P_5\}$ and $[\omega]$ is a primitive integral class. Assume the largest weight is $5$. Then the following 3 conditions are equivalent.
\begin{enumerate}
\item $c_1(M)=3[\omega]$.
\item The largest weight is from $P_0$ to $P_5$ and is equal to $\frac{1}{2}\big(\phi(P_5)-\phi(P_0)\big)$.
\item  The largest weight is from $P_1$ to $P_3$, is from $P_2$ to $P_4$,  and is equal to $\phi(P_3)-\phi(P_1) =\phi(P_4)-\phi(P_2)$; moreover, $\phi(P_1)-\phi(P_0)=\phi(P_5)-\phi(P_4)$.
\end{enumerate}
Under any one condition, $(2)$ and $(3)$ are the only ways the largest weight occurs and  $|\pm 5|$ occurs with multiplicity one at the corresponding fixed points. 
\end{theorem 2}

\begin{proof}
By Proposition~\ref{c1+l} and Lemma~\ref{n02}, (1) implies at least one of  (2) and
\begin{itemize}
\item [(3')]  The largest weight is from $P_1$ to $P_3$ and is equal to $\phi(P_3)-\phi(P_1)$, or is from $P_2$ to $P_4$ and is equal to $\phi(P_4)-\phi(P_2)$, or both. 
\end{itemize}
By Lemmas~\ref{w} and \ref{w'}, (2) implies (3). By Lemma~\ref{lw1}, (3') and (1) 
 imply (2).      
By Proposition~\ref{c1+l}, if (2) or (3) holds, then (1) holds. 
 By Proposition~\ref{c1+l},  $|\pm 5|$ occurs with multiplicity one at the corresponding fixed points. 
\end{proof}

\section{determining the ring $H^*(M; \Z)$, all the weights and the total Chern class $c(M)$ --- proof of Theorems~\ref{thm3} and \ref{thm4}}

In this section, we prove Theorems~\ref{thm3} and \ref{thm4}. 

First in Section 6.1, we determine the sets of weights at all the fixed points for the case that the largest weight is 5. In this case, the assumption is simpler. This proves Theorem~\ref{thm3} (3).
In Section 6.2, for the more general case, assuming the largest weight occurs at the fixed points similarly as in Example~\ref{o}, we determine the sets of weights at all the fixed points. This proves Theorem~\ref{thm4} (3). In Section 6.3, we determine the integral cohomology ring $H^*(M; \Z)$, proving Theorems~\ref{thm3} (1) and \ref{thm4} (1).  In Section 6.4, we determine the total Chern class $c(M)$, proving Theorems~\ref{thm3} (2) and \ref{thm4} (2).

\subsection{Determining all the weights when the largest weight is 5 --- proof of Theorem 3 (3)}
\
\medskip

In this part, under the simple condition that the largest weight is from $P_0$ to $P_5$ and is equal to $5= \frac{1}{2}\big(\phi(P_5)-\phi(P_0)\big)$, we determine the sets of weights at all the fixed points, proving Theorem 3 (3). By Theorem~\ref{thm2}, this condition is equivalent to the topological condition $c_1(M)=3[\omega]$.

\begin{theorem 3(3)}\label{thm3(3)}
Let  $(M, \omega)$ be a compact $10$-dimensional Hamiltonian $S^1$-manifold with moment map $\phi\colon M\to\R$. Assume $M^{S^1}=\{P_0,  P_1, \cdots, P_5\}$ and $[\omega]$ is a primitive integral class.  
If the largest weight is from $P_0$ to $P_5$ and is equal to $5= \frac{1}{2}\big(\phi(P_5)-\phi(P_0)\big)$, then the sets of weights at all the fixed points are:
$$P_0\colon \{1, 4, 2, 3, 5\},\,\,\, P_5\colon -\{1, 4, 2, 3, 5\},$$
$$P_1\colon \{-1, 1, 5, 4, 3\},\,\,\, P_4 \colon -\{-1, 1, 5, 4, 3\},$$
$$P_2\colon \{-4, -1, 1, 5, 2\},\,\,\, P_3 \colon -\{-4, -1, 1, 5, 2\}.$$ 
So the sets of weights at the fixed points are the same as those on $\mathcal O$ in Example~\ref{o}. In particular, there is exactly one weight between any pair of fixed points.
\end{theorem 3(3)}

\begin{proof}
By Lemma~\ref{w}, the sets of weights at $P_0$ and $P_5$ are as claimed. By 
Lemma~\ref{w'}, the sets of weights at $P_1$ and $P_4$ are as claimed. 

We now find the set of weights at $P_3$. 
By Lemma~\ref{w} (2), we have the symmetry for the moment map  images of the fixed points; the weight $-4$ at $P_5$ can only divide $\phi(P_5)-\phi(P_3)$, so 
 there is a weight $4$ between $P_3$ and $P_5$, so $4$ is a positive weight at $P_3$. 
Since $-1$ is a weight at $P_4$ with multiplicity 1, $1$ is a weight at $P_3$ with multiplicity 1 by Lemma~\ref{alcx}. Hence the set of positive weights at $P_3$ is $\{1, 4\}$. 
By Lemma~\ref{w'}, $5$ is a weight from $P_1$ to $P_3$.
By Lemmas~\ref{<} and \ref{w} (2), except $-5$, the other two negative weights at $P_3$  are among $\{-1, -2, -3\}$ (with possible multiplicities).  By Lemma~\ref{mod}, the weights at the normal space of the $M^{\Z_5}$ isotropy manifold at $P_1$, $\{-1, 1, 4, 3\}$,
is equal modulo $5$ to the weights at the normal space at $P_3$. 
So we have the following possibilities, where the weights at $P_3$ are on the left which we need to find, and the weights at $P_1$ which we know are on the right.
\begin{itemize}
\item $1 = 1$, $4=-1$, $-1=4$, $-2 = 3$,  $\mod 5$.
\item  $1 = 1$, $4= 4$, $-1 = -1$, $-2 = 3$,   $\mod 5$.
\end{itemize}
Both cases yield the set of weights at $P_3 \colon \{-2, -5, -1, 1, 4\}$.
By the symmetry in Lemma~\ref{w} (2), the set of weights at $P_2$ is $ -\{-2, -5, -1, 1, 4\}$.

Using the sets of weights at the fixed points, Lemmas~\ref{w} (2), \ref{k|} and \ref{<}, and symmetry, we can see that if we exclude the following ambiguity, we will prove that
there is exactly one weight between any pair of fixed points. The ambiguity is that there is a weight $2$ between $P_2$ and $P_3$, and a weight $2$ between $P_0$ and $P_5$, in this case, the $M^{\Z_2}$ would contain $P_0$, $P_2$, $P_3$ and $P_5$, and the smallest weight $2$ on $M^{\Z_2}$ contradict to Lemma~\ref{alcx}. Hence there is a weight 1 between $P_2$ and $P_3$, a weight $2$ between $P_0$ and $P_3$, and a weight $2$ between $P_2$ and $P_5$.
\end{proof}

\subsection{Determining  the sets of weights at all the fixed points --- proof of Theorem 4 (3)}

\
\medskip

In this part, for the general case, assuming the largest weight occurs at the fixed points as in Example~\ref{o},  we determine the sets of weights at all the fixed points. 

 First, we look at the following lemma on the generators of the integral cohomology ring of the manifold, which we will use in our arguments on determining the weights.

\begin{lemma}\label{ring}
Let  $(M, \omega)$ be a compact $10$-dimensional  Hamiltonian $S^1$-manifold with moment map $\phi\colon M\to\R$. 
Assume $M^{S^1} = \{P_0,  P_1, \cdots, P_5\}$ and $[\omega]$ is a primitive integral class. Then the integral cohomology ring $H^*(M; \Z)$ is generated by the following 
 $\alpha_i\in H^{2i}(M; \Z)$'s for $0\leq i\leq 5$:
\begin{equation}\label{ai}
\alpha_i = \frac{\Lambda_i^-}{\prod_{j=0}^{i-1}\big(\phi(P_j)-\phi(P_i)\big)}[\omega]^i,
\end{equation}
where $\Lambda_i^-$ is the product of the negative weights at $P_i$. 
\end{lemma}
\begin{proof}
 Let $\ut$ be the integral class as in Lemma~\ref{ut}, and  $\big\{\Tilde\alpha_i\in H^{2i}_{S^1}(M; \Z)\,|\, 0\leq i\leq 5\big\}$ be the basis as in Proposition~\ref{equibase}. By Corollary~\ref{cor},
\begin{equation}\label{equi}
\prod_{j=0}^{i-1}\big(\ut +\phi(P_j)t\big) = a_i \Tilde\alpha_i, \,\,\,\mbox{where $a_i\in\Z$}. 
\end{equation}
Restricting (\ref{equi}) to the fixed point $P_i$, we get
$$\prod_{j=0}^{i-1}\big(\phi(P_j)-\phi(P_i)\big) = a_i\Lambda_i^-.$$
Restricting (\ref{equi}) to ordinary cohomology, we get
$$[\omega]^i = a_i\alpha_i.$$
These two equations yield (\ref{ai}). By Proposition~\ref{equibase}, 
$\big\{\alpha_i\,|\, 0\leq i\leq 5\big\}$ is a basis of the ring $H^*(M; \Z)$.
\end{proof}

The next two propositions are the main steps towards determining the sets of weights at all the fixed points. The assumption on the largest weight is more complex than that in Theorem 3(3) when the largest weight is 5.

\begin{proposition}\label{wb1}
Let  $(M, \omega)$ be a compact $10$-dimensional effective Hamiltonian $S^1$-manifold with moment map $\phi\colon M\to\R$. Assume $M^{S^1}=\{P_0,  P_1, \cdots, P_5\}$ and $[\omega]$ is a primitive integral class.  Assume the largest weight  is from $P_0$ to $P_5$, from $P_1$ to $P_3$, and from $P_2$ to $P_4$, and is equal to $\frac{1}{2}\big(\phi(P_5)-\phi(P_0)\big) = \phi(P_3)-\phi(P_1)=\phi(P_4)-\phi(P_2)$. Assume $\phi(P_1)-\phi(P_0)=\phi(P_5)-\phi(P_4)$.
Then
there is a weight between any pair of fixed points; in particular, the weight from $P_0$ to $P_2$ is equal to $\phi(P_2)-\phi(P_0)$, similarly, the weight from $P_3$ to $P_5$ is equal to $\phi(P_5) - \phi(P_3)$.
\end{proposition}

\begin{proof}
Let $w\in\N$ be the number representing the largest weight.
By the assumption, we let
\begin{equation}\label{sym1}
\phi(P_5)-\phi(P_4)=\phi(P_1)-\phi(P_0)=a\,\,\mbox{and}\,\, \phi(P_4)-\phi(P_3)=\phi(P_2)-\phi(P_1) = c.
\end{equation}
Then the assumption implies that
$$\phi(P_3)-\phi(P_2)=2a, \,\,\,\mbox{and}\,\,\, w=2a+c.$$
By Lemma~\ref{mod},
\begin{equation}\label{05mod}
\{\mbox{weights at $P_0$}\} = \{\mbox{weights at $P_5$}\} \mod w =2a+c.
\end{equation}
By Proposition~\ref{c1+l}, the absolute values of the other weights at $P_0$ and $P_5$ than $|\pm w|$ are less than $w$.
By Lemma~\ref{10}, $a=\phi(P_1)-\phi(P_0)$ is a weight at $P_0$, it is from $P_0$ to $P_1$, similarly, $a= \phi(P_5)-\phi(P_4)$ is a weight from $P_4$ to $P_5$.  By  (\ref{05mod}), we must have
\begin{equation}\label{05-1}
a = (-a -c) + 2a + c,
\end{equation}
i.e., $-a-c$ is a weight at $P_5$. Similarly, $a+c$ is a weight at $P_0$. 

{\bf Claim 1}: the weight $a+c$ is from $P_0$ to $P_2$; similarly, $a+c$ is from $P_3$ to $P_5$.

Now we prove {\bf Claim 1}.
The values of $\phi(P_i)-\phi(P_0)$ for $i\neq 0$ are $a$, $a+c$, $3a+c$, $3a+2c$ and $4a+2c$.
 Any weight at $P_0$ divides one of these values by Lemma~\ref{<}.
If $a\neq c$, then the weight $a+c$ at $P_0$ only divides $\phi(P_2)-\phi(P_0)=a+c$, so $a+c$ is from $P_0$ to $P_2$. Similarly, $a+c$ is from $P_3$ to $P_5$.
Next, we consider $a=c$. Then the values of $\phi(P_i)-\phi(P_0)$ are $a$, $2a$, $4a$, $5a$ and $6a$, and $w=3a$. If $2a$ is a weight from $P_0$ to $P_2$, then we are done. Next, assume $2a$ is not a weight from $P_0$ to $P_2$.
Then the weight $2a$ at $P_0$ may be from $P_0$ to $P_3$
(since $2a|4a$) or from $P_0$ to $P_5$ (since $2a|6a$). First assume $2a$ is from $P_0$ to $P_5$.
Consider the $M^{\Z_a}$ containing $P_0$, then it is at least 6-dimensional since it contains the weights $a$, $2a$ and $3a$ at $P_0$, it cannot be more than $6$-dimensional since otherwise  by (\ref{05mod}),  the set of weights at $P_0$ would be
\begin{equation}\label{contr0}
P_0\colon \{a, 2a, a, 2a, 3a\},
\end{equation} 
which contradicts to Lemma~\ref{mult}  (we may let $a=1$ by effectiveness of the action). 
This $M^{\Z_a}$ first contains $P_0$, $P_1$, $P_5$ and $P_4$; since there is a weight $3a$ between $P_1$ and $P_3$, and between $P_2$ and $P_4$ by the assumption, the 
$M^{\Z_a}$ contains all the 6 fixed points. Since $M^{\Z_a}$ is $6$-dimensional and there are already 3 weights at $P_0$, in $M^{\Z_a}$, there cannot be any weight between $P_0$ and $P_2$.  But $P_2$ has index at least $2$ in $M^{\Z_a}$, so there is a weight $a$ between $P_1$ and $P_2$ (since $\phi(P_2)-\phi(P_1)=c=a$). There cannot be a double weight $a$ between $P_1$ and $P_2$ since there are already 3 weights at $P_1$. So $P_2$ has index 2 in 
$M^{\Z_a}$. Then in $M^{\Z_a}$, there are 2 index 2 fixed points $P_1$ and $P_2$ with  weight $-a$, but $a$ occurs once at $P_0$, contradicting to Lemma~\ref{alcx}.  This shows that $2a$ cannot be from $P_0$ to $P_5$.  Next assume $2a$ is from $P_0$ to $P_3$ ($2a$ is not from $P_0$ to $P_2$). Similarly, consider the 
$M^{\Z_a}$ containing $P_0$, which also contains all the fixed points and is 6-dimensional, we draw the weight relation between the fixed points, also applying symmetry, we get the same contradiction as above. Hence $2a$ cannot be from $P_0$ to $P_3$. This finishes the proof of {\bf Claim 1}.

{\bf Claim 2}: there is a weight $\frac{c}{k}$ from $P_1$ to $P_2$ with $1\leq k\in\N$; similarly,
there is a weight $\frac{c}{k}$ from $P_3$ to $P_4$.

We now prove {\bf Claim 2}. By {\bf Claim 1}, $-a-c$ is a weight at $P_2$, it is between $P_0$ and $P_2$. Let  $-w_{21}$ be the other negative weight at $P_2$. Then the $M^{\Z_{w_{21}}}$ containing $P_2$ must contain at least one fixed point below $P_2$. By Lemma~\ref{other}, 
$-w_{21}$ must be a weight between $P_2$ and $P_1$; so $w_{21}$ is a factor of $\phi(P_2)-\phi(P_1)=c$, namely $\frac{c}{k} $ for some $1\leq k\in\N$. By symmetry, there is a weight 
$\frac{c}{k}$ from $P_3$ to $P_4$.

Let 
$$P_0\colon \{a, a+c, x, y, 2a+c\}$$
be the set of weights at $P_0$, where $x$ and 
$y$ are unknown. 
By Lemmas~\ref{other} and \ref{mult} (3), each of $x$ and $y$ has multiplicity 1 at $P_0$, and 
is from $P_0$ to $P_3$, or from $P_0$ to $P_4$, or from $P_0$ to $P_5$, but is less than $2a+c$ by Proposition~\ref{c1+l}.
By symmetry, the set of weights at $P_5$ is
$$P_5\colon \{-2a-c, -y, -x,  -a-c, -a\},$$
Except the known relation (\ref{05-1}) in (\ref{05mod}), since each of $x$ and $y$ has multiplicity 1,  the other relation in  (\ref{05mod}) is
\begin{equation}\label{xy1}
x + y = 2a +c.
\end{equation}

{\bf Claim 3}: $x$ and $y$ cannot be both from $P_0$ to $P_5$.

Now we prove {\bf Claim 3}.  If $x$ and $y$ are both from $P_0$ to $P_5$, then by Lemma~\ref{k|}, we may let
$$x=\frac{4a+2c}{k_1} \,\,\,\mbox{and}\,\,\, y = \frac{4a+2c}{k_2},$$
where $2 < k_1\in\N$, $2 < k_2\in\N$ and $k_1\neq k_2$. 
Then (\ref{xy1}) yields
$$\frac{1}{k_1} + \frac{1}{k_2} =\frac{1}{2}.$$
Since $k_1\neq k_2$, the only possibility is, up to permutation of $k_1$ and $k_2$, 
$$k_1 = 3, \,\,\,\mbox{and}\,\,\,  k_2=6.$$ 
Then the smallest weight on $M^{\Z_y}$ containing $P_0$ is $y$, is from $P_0$ to $P_5$, which has index bigger than 2 in $M^{\Z_y}$, contradicting to Lemma~\ref{alcx} (used on $M^{\Z_y}$).

{\bf Claim 4}: $x$ and $y$ cannot be both from $P_0$ to $P_3$; symmetrically, they cannot be both from $P_2$ to $P_5$.

We now prove {\bf Claim 4}.  Suppose $x$ and $y$ are both from $P_0$ to $P_3$. By Lemma~\ref{k|}, we may let 
$$x=\frac{3a+c}{k_1} \,\,\,\mbox{and}\,\,\, y = \frac{3a+c}{k_2},$$
where $2 \leq k_1\in\N$, $2 \leq k_2\in\N$ and $k_1\neq k_2$. Together with {\bf Claims 1} and {\bf 2} and the assumption, we have the sets of weights at $P_5$ and $P_3$ as follows.
$$P_5\colon \Big\{-a, -(a+c), -\frac{3a+c}{k_1}, -\frac{3a+c}{k_2}, -(2a+c)\Big\}.$$
$$P_3\colon \Big\{-\frac{3a+c}{k_1}, -\frac{3a+c}{k_2}, -(2a+c), \frac{c}{k}, a+c\Big\}.$$
Similar to that of $P_3$, the set of weights at $P_2$ is
$$P_2\colon \Big\{- (a+c), -\frac{c}{k}, (2a+c), \frac{3a+c}{k_1}, \frac{3a+c}{k_2}\Big\}.$$
By Proposition~\ref{c1+l}, we have
\begin{equation}\label{eqc1}
c_1(M)=3[\omega].
\end{equation}
Then by Lemma~\ref{sub} for $P_2$ and $P_3$, we get 
$$\Gamma_2-\Gamma_3 = 3 (2a),$$
where $\Gamma_2$ and $\Gamma_3$ are respectively the sums of weights at $P_2$ and $P_3$.
From this and (\ref{xy1}),  we get
$$ k =1.$$
By Lemma~\ref{ring},
$$\alpha_2 =\frac{1}{k}[\omega]^2,$$
$$\alpha_3 =\frac{1}{k_1k_2}\frac{3a+c}{2a}[\omega]^3,\,\,\,\mbox{and}$$
$$\alpha_5  = \frac{1}{2k_1k_2}\frac{3a+c}{3a+2c}[\omega]^5.$$
By Poincar\'e duality $\alpha_2\alpha_3 = \alpha_5$. So 
$$ka=3a+2c,$$
which contradicts to $k=1$. This proves {\bf Claim 4}.

 {\bf Claim 5}: $x$ and $y$ cannot be both from $P_0$ to $P_4$; symmetrically, they cannot be both from $P_1$ to $P_5$.

We now prove  {\bf Claim 5}.   Suppose $x$ and $y$ are both from $P_0$ to $P_4$. By Lemma~\ref{k|}, let 
$$x=\frac{3a+2c}{k_1} \,\,\,\mbox{and}\,\,\, y = \frac{3a+2c}{k_2},$$
where $2 \leq k_1\in\N$, $2 \leq k_2\in\N$ and $k_1\neq k_2$.
Then by symmetry the set of weights at $P_5$ is
$$P_5\colon\Big\{-2a-c, -\frac{3a+2c}{k_2}, -\frac{3a+2c}{k_1},  -a-c, -a \Big\}.$$
By the information above, {\bf Claim 2} and the assumption,  the set of weights at $P_4$ is
$$P_4\colon \Big\{-\frac{3a+2c}{k_1}, -\frac{3a+2c}{k_2}, -2a-c,  -\frac{c}{k}, a\Big\}.$$
Similarly,  the set of weights at $P_1$ is
$$P_1\colon - \Big\{-\frac{3a+2c}{k_1}, -\frac{3a+2c}{k_2}, -2a-c,  -\frac{c}{k}, a\Big\}.$$
By (\ref{eqc1}), Lemma~\ref{sub} for $P_1$ and $P_4$, and (\ref{xy1}), we get
$$\Gamma_1 - \Gamma_4 = 3(2a+2c),$$
from which we get
$$k=1.$$
By Lemma~\ref{ring},
$$\alpha_1 = [\omega],$$
$$\alpha_4 = \frac{1}{kk_1k_2}\frac{3a+2c}{2a+2c}[\omega]^4,$$
$$\alpha_5 = \frac{1}{2k_1k_2}\frac{3a+2c}{3a+c}[\omega]^5.$$
 By Poincar\'e duality, $\alpha_1\alpha_4 = \alpha_5$, which implies
\begin{equation}\label{eqk}
 k (a+c) = 3a+c,
\end{equation}
which contradicts to $k=1$.

 {\bf Claim 6}: up to permutation of $x$ and $y$, it cannot be that $x$ is from $P_0$ to $P_3$ and $y$ is from $P_0$ to $P_5$.

We now prove {\bf Claim 6}. Suppose $x$  is from $P_0$ to $P_3$, and $y$ is from $P_0$ to $P_5$.
By Lemma~\ref{k|}, let
$$x=\frac{3a+c}{k_1} \,\,\,\mbox{and}\,\,\, y = \frac{4a+2c}{k_2},$$
where $2 \leq k_1\in\N$ and $3 \leq k_2\in\N$. Then (\ref{xy1}) yields that
$$\Big(\frac{3}{k_1}+\frac{4}{k_2} - 2\Big) a + \Big(\frac{1}{k_1}+\frac{2}{k_2}-1\Big) c = 0.$$
We can check that the possible integers $k_1\geq 2$ and $k_2\geq 3$ which make this to be true are
possibly:
\begin{itemize}
\item [(i)]   $k_1=2$, $k_2 = 5$, and $3a=c$.
\item [(ii)]  $k_1=2$, $k_2 = 6$, and $a=c$.
\item [(iii)]   $k_1=2$, $k_2 = 7$, and $a=3c$.
\item [(iv)]  $k_1=4$, $k_2=3$, and $a=c$.
\end{itemize}
For (ii) and (iv), we get the set of weights at $P_0$ as in (\ref{contr0}), which is not possible. 
Now we consider (i). We can write the set of weights at $P_0$:
$$P_0\colon \{a, 4a, 3a, 2a, 5a\}.$$
By effectiveness of the action, we get $a=1$. So the set of weights at $P_0$ is
$$P_0\colon \{1, 4, 3, 2, 5\},$$
where  $4$ is from $P_0$ to $P_2$, and $2$ is from $P_0$ to $P_5$.
By symmetry, the set of weights at $P_5$ is
$$P_5\colon \{-5, -2, -3, -4, -1\}.$$
The $M^{\Z_2}$ containing $P_0$ is 4-dimensional and it contains $P_0$, $P_2$, $P_3$ and $P_5$,
the smallest weight $2$ on $M^{\Z_2}$ contradicts to Lemma~\ref{alcx}. Hence (i) is not possible.
Now we consider (iii). The set of weights at $P_0$ is
$$P_0\colon \{3c, 4c, 5c, 2c, 7c\}.$$
By effectiveness of the action, $c=1$. So the sets of weights at $P_0$ and at $P_5$  are
$$P_0\colon \{3, 4, 5, 2, 7\} \,\,\,\mbox{and}\,\,\, P_5\colon \{-7, -2, -5, -4, -3\}.$$
The same argument as the last case leads to a contraction. Hence (iii) is not possible.

{\bf Claim 7}: up to permutation of $x$ and $y$, it cannot be that $x$ is from $P_0$ to $P_4$, and $y$ is from $P_0$ to $P_5$.

We now prove {\bf Claim 7}. Suppose $x$  is from $P_0$ to $P_4$, and $y$ is from $P_0$ to $P_5$.
Then let
$$x=\frac{3a+2c}{k_1} \,\,\,\mbox{and}\,\,\, y = \frac{4a+2c}{k_2},$$
where $2 \leq k_1\in\N$ and $3 \leq k_2\in\N$. Then (\ref{xy1}) yields that
\begin{equation}\label{45}
\Big(\frac{3}{k_1}+\frac{4}{k_2} - 2\Big) a + \Big(\frac{2}{k_1}+\frac{2}{k_2}-1\Big) c = 0.
\end{equation}
We can check that the possible integers $k_1\geq 2$ and $k_2\geq 3$ which make this to be true are
possibly:
\begin{itemize}
\item [(i)]  $k_1=2$, $k_2\geq 9$.
\item [(ii)] $k_1=3$, $k_2=5$, and $c=3a$.
\item [(iii)] $k_1=5$, $k_2=3$, and $a=c$.
\end{itemize}
As before, we can exclude (iii). Now we consider (i). From (\ref{45}), we get
$$a=\frac{4c}{k_2-8}.$$
Then the set of weights at $P_0$ is
$$P_0\colon \Big\{\frac{4}{k_2-8}c,\, \frac{k_2-4}{k_2-8}c,\, \frac{k_2-2}{k_2-8}c,\, \frac{2}{k_2-8}c,\, \frac{k_2}{k_2-8}c\Big\}.$$
The weight $\frac{4}{k_2-8}c$ is from $P_0$ to $P_1$, and the weight $\frac{2}{k_2-8}c \triangleq l$ is from
$P_0$ to $P_5$. The $M^{\Z_l}$ containing $P_0$ contains at least $P_0$, $P_1$, $P_4$ and $P_5$, where the smallest weight is from $P_0$ to $P_5$, contradicting to Lemma~\ref{alcx}. Hence (i) is not possible.  Now we consider (ii). The set of weights at $P_0$ is
$$P_0\colon \{a, 4a, 3a, 2a, 5a\}.$$
By effectiveness of the action, $a=1$. Note that the weight $4$ is from $P_0$ to $P_2$, and  $2$ is from $P_0$ to $P_5$. 
 As the last case, the smallest weight on $M^{\Z_2}$ containing $P_0$, which is 4-dimensional, contradicts to Lemma~\ref{alcx}. Hence (ii) is not possible. 

By {\bf Claims 3-7}, we get

\noindent {\bf Claim 8}:  up to permutation of $x$ and $y$, $x$ is from $P_0$ to $P_3$ and $y$ is from $P_0$ to $P_4$; symmetrically,
$x$ is from $P_2$ to $P_5$ and $y$ is from $P_1$ to $P_5$.

 {\bf Claim 9}: there is a weight between $P_2$ and $P_3$, and there is a weight between $P_1$ and $P_4$.

We now prove {\bf Claim 9}. There should be 5 weights at each fixed point. From the above, there are 5 weights at $P_0$, respectively from $P_0$ to $P_i$ for all different $i\neq 0$, similarly, the 5 weights at $P_5$ are between $P_5$ and different
$P_i$ for $i\neq 5$. For $P_1$, we know it has a weight with $P_0$, with $P_2$, with $P_3$ and with $P_5$. Similarly, for each of $P_2$,  $P_3$ and $P_4$, we already have 4 weights. The positive weight at $P_1$ that we do not know can only be either from $P_1$ to $P_3$ or from $P_1$ to $P_4$. If it is from $P_1$ to $P_4$, then the unknown positive weight at $P_2$ must be from $P_2$ to $P_3$, then we are done. Assume the unkown positive weight at $P_1$ is from $P_1$ to $P_3$, denote it as $x_1$. Then symmetrically 
$x_1$ is from $P_2$ to $P_4$. We argue below that this is impossible. We write the set of weights at 
$P_1$ and $P_3$:
$$P_1\colon \big\{-a, \,\frac{c}{k}, \,2a+c,\,x_1,\, y\big\},$$
$$P_3\colon \big\{- x,\, -2a-c,\, -x_1,\, \frac{c}{k},\, a+c\big\}.$$
By Lemma~\ref{sub} for $P_1$ and $P_3$, (\ref{xy1}) and (\ref{eqc1}), we get
$$x_1 = \frac{2a+c}{2}.$$
Then the $M^{\Z_{x_1}}$ containing $P_1$ does not contain $P_0$, i.e., $P_1$ is the minimum in $M^{\Z_{x_1}}$.
The $M^{\Z_{x_1}}$, if it does not contain $P_2$, then it contains the index 4 fixed point $P_3$ as the closest fixed point,  contradicting to Lemma~\ref{ind}; if it contains
$P_2$, then  $M^{\Z_{x_1}}$ is at least 6-dimensional, $P_2$ is the only index 2 fixed point in $M^{\Z_{x_1}}$, the smallest weight $x_1$ in
$M^{\Z_{x_1}}$ contradicts to Lemma~\ref{alcx}.
\end{proof}

\begin{proposition}\label{wb2}
Let  $(M, \omega)$ be a compact $10$-dimensional effective Hamiltonian $S^1$-manifold with moment map $\phi\colon M\to\R$. Assume $M^{S^1}=\{P_0,  P_1, \cdots, P_5\}$ and $[\omega]$ is a primitive integral class.  Assume the largest weight  is from $P_0$ to $P_5$, from $P_1$ to $P_3$, and from $P_2$ to $P_4$, and is equal to $\frac{1}{2}\big(\phi(P_5)-\phi(P_0)\big) = \phi(P_3)-\phi(P_1)=\phi(P_4)-\phi(P_2)$. Assume $\phi(P_1)-\phi(P_0)=\phi(P_5)-\phi(P_4)$. If
there is a weight between any pair of fixed points,  then the sets of weights at all the fixed points are:
$$P_0\colon \Big\{a,\, a+c,\, a+\frac{c}{3},\, a+\frac{2}{3}c,\, 2a+ c\Big\},\,\,\,  P_5\colon\Big\{-2a-c,\, -a-\frac{2}{3}c,  \, -a-\frac{c}{3},\, -a-c,\, -a\Big\},$$
$$P_1\colon \Big\{-a,\, \frac{c}{3}, \, 2a+ c,\, a+c,\, a+\frac{2}{3}c\Big\}, \,\,\,  P_4\colon \Big\{-a-\frac{2}{3}c,\, -a-c,\, -2a-c,\, - \frac{c}{3},\, a\Big\},$$
$$P_2\colon \Big\{-a-c,\, - \frac{c}{3},\, a, \, 2a+c,\, a+\frac{c}{3}\Big\}, \,\,\,  P_3 \colon \Big\{-a-\frac{c}{3},\, -2a-c,\, -a,\, \frac{c}{3}, \, a+c\Big\},$$
where $a = \phi(P_1)-\phi(P_0)$ and $c = \phi(P_2)-\phi(P_1)$.
\end{proposition}

\begin{proof}
By the assumption, we have (\ref{sym1}). Then the largest weight is $w=2a+c$.
By the assumption, Lemma~\ref{10} and Proposition~\ref{wb1}, the set of weights at $P_0$ can be written
$$P_0\colon \Big\{a, a+c, \frac{3a+c}{k_1}, \frac{3a+2c}{k_2}, 2a+c\Big\},$$
and by symmetry the set of weights at $P_5$ is
$$P_5\colon \Big\{-2a-c, -\frac{3a+2c}{k_2},  -\frac{3a+c}{k_1}, -a-c, -a\Big\},$$
where $2\leq k_1\in\N$ and $2\leq k_2\in\N$, and each of $\frac{3a+c}{k_1}$ and $\frac{3a+2c}{k_2}$
has multiplicity 1.
By (\ref{05mod}), we have (\ref{05-1}) (i.e. $ a = -a - c + (2a+c)$) 
and
\begin{equation}\label{k12}
\frac{3a+c}{k_1} +  \frac{3a+2c}{k_2} = 2a +c.
\end{equation}
Equation (\ref{k12}) gives
$$\Big(\frac{3}{k_1} + \frac{3}{k_2} - 2\Big) a + \Big(\frac{1}{k_1}+\frac{2}{k_2} -1\Big)c = 0.$$
We can check that the integers $k_1\geq 2$ and $k_2\geq 2$ which make this to be true are possibly as follows:
\begin{itemize}
\item    $k_1 > 6$ and $k_2 = 2$.
\item   $k_1 = k_2 = 3$.
\item    $k_1 = 2$ and $k_2 = 5$, in which case $a=c$.
\end{itemize}
By the assumption, symmetry and Proposition~\ref{wb1}, we can write the sets of weights at $P_1$ and $P_3$:
$$P_1\colon \Big\{-a, \frac{c}{k}, 2a+c, \frac{2a+2c}{l_1}, \frac{3a+2c}{k_2}\Big\},$$
$$P_3\colon \Big\{-\frac{3a+c}{k_1}, -2a-c, -\frac{2a}{l_2}, \frac{c}{k}, a+c \Big\},$$
where $2\leq l_1\in\N$, $1\leq l_2\in\N$, and $1\leq k\in\N$. 
By Lemma~\ref{sub} for $P_1$ and $P_3$, (\ref{k12}) and (\ref{eqc1}), we get
\begin{equation}\label{l12}
\frac{2a+2c}{l_1} + \frac{2a}{l_2} = 2a + c.
\end{equation}
Equation (\ref{l12}) gives
$$\Big(\frac{2}{l_1}+\frac{2}{l_2} -2\Big) a + \Big(\frac{2}{l_1}-1\Big) c = 0.$$
We can check that the integers $l_1\geq 2$ and $l_2\geq 1$ which make this to be true are possibly
as follows:
\begin{itemize}
\item  $l_1 = l_2 = 2$.
\item  $l_1 > 2$ and $l_2 = 1$.
\end{itemize}
By symmetry, the set of weights at $P_4$ is the negative of the set of weights at $P_1$, and the set of weights at $P_2$ is the negative of the set of weights at $P_3$.
By Lemma~\ref{ring},
$$\alpha_5 = \frac{1}{2k_1k_2}[\omega]^5,$$
$$\alpha_1 = [\omega], \,\, \alpha_4 = \frac{1}{kl_1k_2}[\omega]^4,$$
$$\alpha_2= \frac{1}{k}[\omega]^2, \,\,\,\mbox{and}\,\,\, \alpha_3 = \frac{1}{l_2k_1}[\omega]^3.$$
By Poincar\'e duality, $\alpha_1\alpha_4 = \alpha_5$ and $\alpha_2\alpha_3=\alpha_5$, from which   we get
$$kk_2l_1=2k_1k_2 \,\,\,\mbox{and}\,\,\, kk_1l_2 = 2 k_1k_2.$$ 
So
$$\frac{l_1}{l_2} = \frac{k_1}{k_2}.$$
By the above possibilities  for $k_1$, $k_2$, $l_1$ and $l_2$, we have either
\begin{itemize}
\item $k_1=k_2 = 3$ and $l_1=l_2 = 2$, or
\item $k_2 = 2$, $k_1 > 6$, $l_2 = 1$, $l_1 > 2$ and $k_1 = 2l_1.$
\end{itemize}
For the second case, (\ref{k12}) and (\ref{l12}) yield that $a=c$, $k_1=8$ and $l_1=4$; and the set of weights at $P_0$ is $\Big\{a, 2a, \frac{a}{2}, \frac{5a}{2}, 3a\Big\}$. By effectiveness of the action, we get $a=2$,  then the smallest weight $1$ on $M$ is from $P_0$ to $P_3$, contradicting to Lemma~\ref{alcx}. Hence we only  have the case that 
$$k_1=k_2 = 3 \,\,\,\mbox{and}\,\,\, l_1=l_2 = 2.$$ 
Then we get 
$$k=3.$$
Then we get the sets of weights at all the fixed points as claimed.
\end{proof}

Propositions~\ref{wb1} and \ref{wb2} yield Theorem~\ref{thm4} (3) in the Introduction.

\subsection{Determining  the ring $H^*(M; \Z)$ --- proof of Theorems 3 (1) and 4 (1)}

\
\medskip

In this part,  we determine the ring $H^*(M; \Z)$, proving Theorems 3 (1) and 4 (1).

Note that  each of the generators  $\alpha_i$'s of the ring $H^*(M; \Z)$ in Lemma~\ref{ring} is  determined by the product $\Lambda_i^-$ of the negative weights at $P_i$ and its relation with $\prod_{j=0}^{i-1}\big(\phi(P_j)-\phi(P_i)\big)$, we do not need to know the individual weights at the  fixed point $P_i$. 
Moreover, we only need to know $\alpha_1$, $\alpha_2$ and $\alpha_5$ or $\alpha_1$, $\alpha_3$ and $\alpha_5$.

\begin{theorem 4(1')}\label{thm4(1)}
Let  $(M, \omega)$ be a compact $10$-dimensional effective Hamiltonian $S^1$-manifold with moment map $\phi\colon M\to\R$. 
Assume $M^{S^1} = \{P_0,  P_1, \cdots, P_5\}$ and $[\omega]$ is a primitive integral class. If 
$$\Lambda_2^- = \frac{1}{3}\prod_{j=0}^{1}\big(\phi(P_j)-\phi(P_2)\big) \,\,\, \Big(\mbox{or}\,\,\, \Lambda_3^- = \frac{1}{6}\prod_{j=0}^{2}\big(\phi(P_j)-\phi(P_3)\big)\Big), \,\,\,\mbox{and}$$ 
$$\Lambda_5^- = \frac{1}{18}\prod_{j=0}^{4}\big(\phi(P_j)-\phi(P_5)\big),$$
then the ring $H^*(M; \Z)$ is generated by the following elements:
$$1, \,\,\, [\omega], \,\,\,\frac{1}{3}[\omega]^2, \,\,\, \frac{1}{6}[\omega]^3,\,\,\, \frac{1}{18}[\omega]^4,\,\,\, \frac{1}{18}[\omega]^5.$$
\end{theorem 4(1')}

\begin{proof}
By Lemma~\ref{10}, $\Lambda_1^- = \phi(P_0)-\phi(P_1)$.
Then by the assumption and  Lemma~\ref{ring}, we get 3 generators $[\omega]\in H^2(M; \Z)$, 
$\frac{1}{3}[\omega]^2\in H^4(M; \Z)$, and
$\frac{1}{18}[\omega]^5\in H^{10}(M; \Z)$. By Lemma~\ref{ring} and Poincar\'e duality, 
$\frac{1}{6}[\omega]^3\in H^6(M; \Z)$ and $\frac{1}{18}[\omega]^4\in H^8(M; \Z)$ are generators.
So the claim follows by Lemma~\ref{ring}.
\end{proof}

Theorem 3(3) and Theorem 4(1') yield Theorem~\ref{thm3} (1) in the Introduction.
Theorem 4 (3) in the Introduction and Theorem 4(1') yield   Theorem~\ref{thm4} (1) in the Introduction. 

\subsection{Determining the total Chern class $c(M)$ --- proof of Theorems 3 (2) and 4 (2)}
\
\medskip

In this part,  we determine the total Chern class $c(M)$, proving Theorems 3 (2) and 4 (2).

\begin{lemma}\label{extension}
Let  $(M, \omega)$ be a compact $10$-dimensional Hamiltonian $S^1$-manifold with moment map $\phi\colon M\to\R$. Assume $M^{S^1}=\{P_0,  P_1, \cdots, P_5\}$ and $[\omega]$ is a primitive integral class. Assume the ring $H^*(M; \Z)$ is generated by the elements:
$$1, \,\,\, [\omega], \,\,\,\frac{1}{3}[\omega]^2, \,\,\,\frac{1}{6}[\omega]^3, \,\,\,\frac{1}{18}[\omega]^4\,\,\,\frac{1}{18}[\omega]^5.$$
Then  $H^*_{S^1}(M; \Z)$ has a basis $\big\{\Tilde\alpha_i\,|\, 0\leq i\leq 5\big\}$  as follows, where $\ut$ is as in Lemma~\ref{ut}.
$$1,  \,\,\,\Tilde\alpha_1=\ut+\phi(P_0)t,  \,\,\,  \Tilde\alpha_2=\frac{1}{3}\prod_{j=0}^{1}\big(\ut +\phi(P_j)t\big),$$
$$\Tilde\alpha_3=\frac{1}{6}\prod_{j=0}^{2}\big(\ut +\phi(P_j)t\big),  \,\,\,
 \Tilde\alpha_4=\frac{1}{18}\prod_{j=0}^{3}\big(\ut +\phi(P_j)t\big), \,\,\,
 \Tilde\alpha_5=\frac{1}{18}\prod_{j=0}^{4}\big(\ut +\phi(P_j)t\big).$$
\end{lemma}

\begin{proof}
By Proposition~\ref{equibase}, $H^*_{S^1}(M; \Z)$ has a basis $\big\{\Tilde\alpha_i\,|\, 0\leq i\leq 5\big\}$ satisfying the stated properties there. By Corollary~\ref{cor}, there exists $a_i\in\Z$, such that
$$\prod_{j=0}^{i-1}\big(\ut +\phi(P_j)t\big) = a_i \Tilde\alpha_i.$$
Restricting it to ordinary cohomology, we get the basis $\alpha_i$'s in ordinary cohomology:
$$[\omega]^i = a_i\alpha_i.$$
By the assumption on the basis $\big\{\alpha_i\,|\, 0\leq i\leq 5\big\}$ in ordinary cohomology, we get the integers $a_i$'s, then we get the $\Tilde\alpha_i$'s.
\end{proof}

\begin{theorem 4(2)}\label{totalchern}
Let  $(M, \omega)$ be a compact $10$-dimensional  Hamiltonian $S^1$-manifold with moment map $\phi\colon M\to\R$. Assume $M^{S^1}=\{P_0,  P_1, \cdots, P_5\}$ and $[\omega]$ is a primitive integral class.  Assume the largest weight  is from $P_0$ to $P_5$, from $P_1$ to $P_3$, and from $P_2$ to $P_4$, and is equal to $\frac{1}{2}\big(\phi(P_5)-\phi(P_0)\big) = \phi(P_3)-\phi(P_1)=\phi(P_4)-\phi(P_2)$. Assume $\phi(P_1)-\phi(P_0)=\phi(P_5)-\phi(P_4)$.
Then the total Chern class 
$$c(M) = 1 + 3\alpha_1 +13 \alpha_2 +22\alpha_3 +30 \alpha_4 +6 \alpha_5,$$
$$=1+3[\omega]+\frac{1}{3}(13[\omega]^2 + 11[\omega]^3 + 5[\omega]^4 + [\omega]^5).$$
where $\alpha_i\in H^{2i}(M; \Z)$ is a generator for each $1\leq i\leq 5$.
\end{theorem 4(2)}

\begin{proof}
By the assumption, we may let
$$\phi(P_1)-\phi(P_0)=\phi(P_5)-\phi(P_4)= a, \,\,\,\mbox{and}\,\,\, \phi(P_2)-\phi(P_1)=\phi(P_4)-\phi(P_3)= c.$$
Then 
$$\phi(P_3)-\phi(P_2)=2a.$$
By Theorem~\ref{thm4} (3), we have the sets of weights at all the fixed points in terms of $a$ and 
$c$. Since $c$ should be divisible by $3$, we let 
$$c=3b.$$ 
Then
we can write the sets of weights at all the fixed points as follows:
$$P_0\colon \{a, a+3b, a+b, a+2b, 2a+3b\},\,\,\,  P_5\colon \{-2a-3b, -a-2b, -a-b, -a-3b, -a\},$$
$$P_1\colon \{-a, b,  2a+3b, a+3b, a+2b\}, \,\,\,  P_4\colon \{-a-2b, -a-3b, -2a-3b, -b, a\},$$
$$P_2\colon \{-a-3b, -b, a, 2a+3b, a+b\}, \,\,\,  P_3 \colon \{-a-b, -2a-3b, -a, b, a+3b\}.$$
Then we obtain the following data.
$$c_1^{S^1}(M)|_{P_0}=(6a+9b) t, \,\,\,c_1^{S^1}(M)|_{P_1}= (3a+9b) t.$$
$$c_2^{S^1}(M)|_{P_0}=(14a^2+42ab+29b^2) t^2, \,\,\,c_2^{S^1}(M)|_{P_1}=(a^2+16ab+29b^2) t^2,$$
$$c_2^{S^1}(M)|_{P_2}= (a^2 -10ab -10b^2) t^2.$$
$$c_3^{S^1}(M)|_{P_0}= (16a^3 + 72a^2b + 98ab^2 + 39 b^3) t^3, \,\,\,c_3^{S^1}(M)|_{P_1}= (-3a^3-7a^2b + 19ab^2 + 39b^3) t^3,$$
$$c_3^{S^1}(M)|_{P_2}=  (-3a^3 -20a^2b -20ab^2) t^3,\,\,\, c_3^{S^1}(M)|_{P_3}=(3a^3 +20a^2b +20ab^2)  t^3.$$
$$c_4^{S^1}(M)|_{P_0}= (9a^4+54a^3b+ 109a^2b^2 +84ab^3 + 18b^4) t^4, $$
$$c_4^{S^1}(M)|_{P_1}= (-2a^4 -16a^3b -35a^2b^2 -12ab^3 +18 b^4) t^4,$$
$$c_4^{S^1}(M)|_{P_2}= (-2a^4 -8a^3b + a^2b^2 + 18ab^3 + 9b^4) t^4,$$
$$c_4^{S^1}(M)|_{P_3}= (-2a^4 -8a^3b +a^2b^2 +18 ab^3 +9b^4) t^4, $$
$$ c_4^{S^1}(M)|_{P_4}= (-2a^4 -16a^3b -35a^2b^2 -12ab^3 +18 b^4) t^4.$$

By Theorem 4 (1), we have the basis $\{\Tilde\alpha_i\,|\, 0\leq i\leq 5\}$   of $H^*_{S^1}(M; \Z)$ as in Lemma~\ref{extension}. Using this basis and the degree of $c_1^{S^1}(M)$, 
we can write $c_1^{S^1}(M) = a_0t + a_1\Tilde\alpha_1$, where 
$a_0, a_1\in\Z$. Restricting it to $P_0$ and $P_1$ respectively, we get 
$$c_1^{S^1}(M) = (6a+9b) t+3\Tilde\alpha_1.$$ 
Restricting it to ordinary cohomology, we get
$$c_1(M) = 3\alpha_1.$$
Similarly, we can write $c_2^{S^1}(M) = b_0t^2 + b_1t\Tilde\alpha_1 + b_2\Tilde\alpha_2$. Respectively restricting it to $P_0$, $P_1$ and $P_2$, we get 
$$c_2^{S^1}(M) = (14a^2 + 42ab + 29 b^2) t^2 + (13a + 26b) t\Tilde\alpha_1+13\Tilde\alpha_2.$$ 
Restricting it to ordinary cohomology, we get
$$c_2(M) = 13\alpha_2.$$
Similarly, we have $c_3^{S^1}(M) = c_0t^3 + c_1t^2\Tilde\alpha_1 + c_2t\Tilde\alpha_2+c_3\Tilde\alpha_3$. Respectively restricting it to $P_0$, $P_1$, $P_2$ and $P_3$, we get
$$c_3^{S^1}(M) = (16a^3 +72a^2b + 98ab^2 +39b^3)t^3 + (19a^2 + 79ab + 79 b^2)t^2\Tilde\alpha_1 + (44a+66b) t\Tilde\alpha_2+22\Tilde\alpha_3.$$ 
Restricting it to ordinary cohomology, we get
$$c_3(M) = 22\alpha_3.$$
We similarly write $c_4^{S^1}(M) = d_0t^4 + d_1t^3\Tilde\alpha_1 + d_2t^2\Tilde\alpha_2+d_3t\Tilde\alpha_3 + d_4\Tilde\alpha_4$. Respectively restricting it to $P_0$, $P_1$, $P_2$, $P_3$ and $P_4$, we get  
$$c_4^{S^1}(M) =(9a^4 +54a^3b +109a^2b^2 +84 ab^3 + 18b^4) t^4+(11a^3 +70a^2b + 144ab^2 +96b^3) t^3\Tilde\alpha_1$$
$$+(41a^2 +123ab + 93 b^2)t^2\Tilde\alpha_2+30(a+2b) t\Tilde\alpha_3+30\Tilde\alpha_4.$$ 
Restricting it to ordinary cohomology, we get
$$c_4(M) = 30 \alpha_4.$$
Since the top Chern number is equal to the number of fixed points,  we have
$$c_5(M) = 6\alpha_5.$$
The last equality in the claim is by Theorem 4 (1).
\end{proof}

Theorem~\ref{thm4} (2) in the Introduction follows from the theorem above.
By Theorem 3(3),  Theorem~\ref{thm3} (2) in the Introduction is a special case of Theorem 4(2) above.


\begin{thebibliography}{99}
\bibitem{GS} L. Godinho and S. Sabatini, {\em New tools for classifying Hamiltonian circle actions with isolated fixed points}, Foundations in Computational Mathematics, 14(2014), 791-860.
\bibitem{H}  A. Hattori, {\em $S^1$ actions on unitary manifolds and quasi-ample line bundles}, J. Fac. Sci. Univ. Tokyo Sect. IA, Math., {\bf 31},
no. 3  (1984), 433-486.
\bibitem{J} D. Jang, {\em  Symplectic periodic flows with exactly three equilibrium points}, Ergod. Th. and Dynam. Sys. (2014), {\bf 34}, 1930-1963.
\bibitem{JT} D. Jang and S. Tolman,  {\em Hamiltonian circle actions on eight dimensional manifolds with minimal fixed sets},  Transformation groups, vol. {\bf 22}, no. 2, 2017, 353-359.
\bibitem{Ka} Y. Karshon, {\em Periodic  Hamiltonian flows on four dimensional manifolds}, Mem. Amer. Math. Soc., {\bf 672}, 1999.
 \bibitem{K} F. Kirwan, {\em Cohomology of Quotients in Symplectic and Algebraic Geometry}, Princeton University Press, 1984.
\bibitem{KT} D. Kotschick and D.K.Thung, {\em The complex geometry of two exceptional flag manifolds}, Annali di Matematica Pura ed Applicata(1923 -)(2020) 199:2227-2241.
\bibitem{L} H. Li, {\em Hamiltonian circle actions with minimal isolated fixed points}, Math. Z., 2023, 304:33.
\bibitem{L2} H. Li, {\em Certain circle actions on K\"ahler manifolds}, International Mathematics Research Notices,  Vol. 2014, No. 18, 5187-5202.
\bibitem{LOS} H. Li, M. Olbermann and D. Stanley, {\em  One connectivity and finiteness of Hamiltonian $S^1$-manifolds
with minimal fixed sets}, J. of Lond. Math. Soc. (2)  {\bf 92}  (2015),  no. 2, 284-310.
\bibitem{LT} H. Li and S. Tolman,  {\em Hamiltonian circle actions with minimal fixed sets},  International Journal of Mathematics {\bf 23}, no. 8 (2012), 1250071.
\bibitem{L'}  P. Li, {\em The rigidity of Dolbeault-type operators and symplectic circle actions}, Proc. Amer. Math. Soc. {\bf 140} (2012), 1987-1995.
\bibitem{Mc} D. McDuff, {\em Some $6$-dimensional Hamiltonian $S^1$ manifolds}, J. of Topology, {\bf 2} (2009), no. 3, 589-623.
\bibitem{Mo} D. Morton, {GKM manifolds with low Betti numbers}, Ph.D thesis, University of Illinois at Urbana-Champaign, 2011.
\bibitem{P} T. Petrie,  {\em Smooth $S^1$ actions on homotopy complex projective spaces and related topics}, Bull. Amer. Math. Soc. {\bf 78} (1972) 105-153.
\bibitem{T} S. Tolman, {\em On a symplectic generalization of Petrie's
conjecture},  Trans. Amer. Math. Soc., {\bf 362} (2010), 3963-3996.
\end{thebibliography}
\end{document}